\documentclass[12pt]{article}
\usepackage[usenames,dvipsnames,svgnames,table]{xcolor} 

\usepackage{rotfloat} 
\usepackage{graphicx}
\usepackage{epsfig}
\usepackage{amssymb, amsmath}
\usepackage{verbatim}
\usepackage{authblk}
\usepackage{kotex}
\usepackage{multirow}
\usepackage[colorlinks]{hyperref} 
\usepackage[colorinlistoftodos, textsize=scriptsize]{todonotes} 

\usepackage[framemethod=TikZ]{mdframed}
\mdfdefinestyle{JL}{%
    linecolor=blue,
    outerlinewidth=1pt,
    roundcorner=2pt,
    innertopmargin=0.1\baselineskip,
    innerbottommargin=0.1\baselineskip,
    innerrightmargin=2pt,
    innerleftmargin=2pt,
    backgroundcolor=orange!100!white}

\newtheorem{theorem}{Theorem}[section]
\newtheorem{lemma}[theorem]{Lemma}
\newtheorem{proposition}[theorem]{Proposition}
\newtheorem{corollary}[theorem]{Corollary}

\newtheorem{example}[theorem]{Example}

\newenvironment{proof}[1][Proof]{\begin{trivlist}
\item[\hskip \labelsep {\bfseries #1}]}{\end{trivlist}}
\newenvironment{definition}[1][Definition]{\begin{trivlist}
\item[\hskip \labelsep {\bfseries #1}]}{\end{trivlist}}
\newenvironment{remark}[1][Remark]{\begin{trivlist}
\item[\hskip \labelsep {\bfseries #1}]}{\end{trivlist}}

\newcommand{\qed}{\nobreak \ifvmode \relax \else
      \ifdim\lastskip<1.5em \hskip-\lastskip
      \hskip1.5em plus0em minus0.5em \fi \nobreak
      \vrule height0.75em width0.5em depth0.25em\fi}

\newcommand{\ech}{\color{black}  ~}    

\DeclareMathOperator{\argmin}{argmin}

\newcommand{\bea}{\begin{eqnarray*}}
\newcommand{\eea}{\end{eqnarray*}}
\newcommand{\bean}{\begin{eqnarray}}
\newcommand{\eean}{\end{eqnarray}}
\newcommand{\baln}{\begin{align}}
\newcommand{\ealn}{\end{align}}
\newcommand{\bdefi}{\begin{definition}}
\newcommand{\edefi}{\end{definition}}
\newcommand{\bi}{\begin{itemize}}
\newcommand{\ei}{\end{itemize}}
\newcommand{\benu}{\begin{enumerate}}
\newcommand{\eenu}{\end{enumerate}}
\newcommand{\ben}{\begin{enumerate}}
\newcommand{\een}{\end{enumerate}}

\newcommand{\V}{{\rm Var}}
\newcommand{\C}{{\rm Cov}}
\newcommand{\sg}{\Sigma}

\newcommand{\what}{\widehat}

\newcommand{\bfX}{{\bf X}}

\newcommand{\bbE}{\mathbb{E}}

\newcommand{\bbP}{\mathbb{P}} 
\newcommand{\bbR}{\mathbb{R}}
\newcommand{\cC}{\mathcal{C}}
\newcommand{\cG}{\mathcal{G}}

\newcommand{\cU}{\mathcal{U}}

\newcommand{\lra}{\longrightarrow}

\newcommand{\calC}{\mathcal{C}}

\begin{document}

\title{Estimating Large Precision Matrices via Modified Cholesky Decomposition}
\author[1]{Kyoungjae Lee}
\author[2]{Jaeyong Lee}
\affil[1]{Department of Applied and Computational Mathematics and Statistics, The University of Notre Dame}
\affil[2]{Department of Statistics, Seoul National University}

\maketitle
\begin{abstract}
We introduce the $k$-banded Cholesky prior for estimating a high-dimensional bandable precision matrix via the modified Cholesky decomposition.  
The bandable assumption is imposed on the Cholesky factor of the decomposition. 
We obtained the P-loss convergence rate under the spectral norm and the matrix $\ell_{\infty}$ norm and the minimax lower bounds. Since the P-loss convergence rate (Lee and Lee (2017)\nocite{lee2017optimal}) is stronger than the posterior convergence rate, the rates obtained are also posterior convergence rates. 
Furthermore, when the true precision matrix is a $k_0$-banded matrix with some finite $k_0$, the obtained P-loss convergence rates coincide with the minimax rates.
The established convergence rates are slightly slower than the minimax lower bounds, but these are the fastest rates for bandable precision matrices among the existing Bayesian approaches. 
A simulation study is conducted to compare the performance to the other competitive estimators in various scenarios.
\end{abstract}


\section{Introduction}
Due to the technical advances, it is common that the number of variables $p$ of the data sets collected in recent years is much larger than their sample size $n$.
Examples of such high-dimensional data sets arise from genomics, climatology, fMRI and neuroimaging, to name just a few. 
In this paper, we concentrate on the estimation of the precision matrix, the inverse of the covariance matrix, for a high-dimensional data.
In the analysis of dependent data, often the estimation of the covariance or precision matrix is a crucial initial step of subsequent analysis,  for example  principle component analysis (PCA), linear/quadratic discriminant analysis and multivariate analysis of variance (MANOVA).

When the number of variables $p$ tends to infinity as $n\lra \infty$ and is even possibly larger than $n$, the traditional sample covariance fails to converge to the true covariance marix (Johnstone and Lu, 2009)\nocite{johnstone2009consistency}.\ech
It becomes necessary to assume restrictive matrix classes to get a consistent estimator under the ultra high-dimensional setting, $\log p=o(n)$.
The restriction on the matrix class includes the sparse, bandable assumption or lower-dimensional structure such as sparse spiked covariance and factor model. 
The minimax convergence rates under the sparsity or bandable assumption on a covariance/precision matrix itself were established by Bickel and Levina (2008a, 2008b)\nocite{bickel2008covariance},  Cai, Zhang and Zhou (2010)\nocite{cai2010optimal}, Cai and Zhou (2012a, 2012b)\nocite{cai2012minimax}\nocite{cai2012optimal}, Xue and Zou (2013)\nocite{xue2013minimax} and Cai, Liu and Zhou (2016)\nocite{cai2016estimating}, to just name a few. 
The convergence rates for precision matrices under the sparsity or bandable assumption via Cholesky decomposition were studied by Bickel and Levina (2008b) and Verzelen (2010)\nocite{verzelen2010adaptive}. 
The convergence rates under lower-dimensional structures of covariance matrix such as factor model (Fan, Fan and Lv, 2008)\nocite{fan2008high} and sparse spiked covariance model (Cai, Ma and Wu, 2015)\nocite{cai2015optimal} were also explored. 
Cai, Liang and Zhou (2015)\nocite{cai2015law} and Fan, Rigollet and Wang (2015)\nocite{fan2015estimation} derived the minimax convergence rates for the functionals of the covariance matrices. 
Cai, Ren and Zhou (2016)\nocite{cai2016review} provided a comprehensive review on the convergence rate for large matrices.

The posterior convergence rates for large covariance or precision matrices have been received attention, but there are only a few works available in high-dimensional settings. 
Banerjee and Ghosal (2015)\nocite{banerjee2015bayesian} showed the posterior convergence rate for the precision matrix under the sparsity assumption. They used a mixture prior for off-diagonal elements of the precision matrix to assign exactly zero.
To estimate bandable precision matrices, Banerjee and Ghosal (2014)\nocite{banerjee2014posterior} utilized the $G$-Wishart prior on the precision matrix and established the posterior convergence rate. 
Xiang, Khare and Ghosh (2015)\nocite{xiang2015high} extended the result of Banerjee and Ghosal (2014) to decomposable graphical models which contains the bandable precision matrices as a special case.
Pati et al. (2014)\nocite{pati2014posterior} considered the posterior convergence rate for covariance estimation via the sparse factor model. They obtained nearly optimal rates,  the minimax rates with $\sqrt{\log n}$ factor when the number of true factors is bounded.
The optimal posterior convergence rate for covariance matrices of sparse spiked covariance model was derived by Gao and Zhou (2015)\nocite{gao2015rate}. 
The above results assumed the ultra high-dimensional setting, $\log p =o(n)$, or its variants. 
Recently, Gao and Zhou (2016)\nocite{gao2016bernstein} derived Bernstein-von Mises theorems for functionals of the covariance matrix as well as its inverse, under conditions such as $p=o(n)$ or $p^3=o(n)$.

Lee and Lee (2017)\nocite{lee2017optimal} proposed a new decision theoretical framework for prior selection and obtained the Bayesian minimax rate of the unconstrained covariance matrix under the spectral norm for all rates of $p$. The Bayesian minimax rates under the Frobenius norm, the Bregman divergence and squared log-determinant loss were also obtained when $p\le \sqrt{n}$ or $p =o(n)$. Lee and Lee (2017) showed that when $p > n/2$, there is no better prior than the point mass prior $\delta_{I_p}$ in terms of the induced posterior convergence rate. Thus, it implies that a certain restriction on the covariance or precision matrix is needed for the consistent estimation.

In this paper, we consider a class of bandable precision matrices via the modified Cholesky decomposition (MCD) under the ultra high-dimensional setting, and derive the P-loss convergence rates under the spectral norm and matrix $\ell_\infty$ norm. 
The bandable assumption is imposed on the lower triangular matrix from the MCD, which is called the Cholesky factor. 
Bickel and Levina (2008b)\nocite{bickel2008regularized} used a similar assumption: their parameter space is a special case of ours.
Our work is also closely related to the works of Banerjee and Ghosal (2014)\nocite{banerjee2014posterior} and Xiang, Khare and Ghosh (2015)\nocite{xiang2015high} for they considered the bandable precision matrices. 
However, we emphasize that the convergence rate obtained in this paper is faster than those obtained in the above papers. To the best of our knowledge, the convergence rate in this paper is the fastest rate for the bandable precision matrices among existing Bayesian methods.
Although our parameter space is not exactly same as that of Banerjee and Ghosal (2014), they are closely related. Proposition \ref{equi_para} describes the relationship between them.  Furthermore, we show the minimax lower bound for precision matrices under the $k$-banded assumption on the Cholesky factor. The lower bounds are derived under the spectral norm as well as matrix $\ell_\infty$ norm. To the best of our knowledge, this is the first result on minimax lower bound for precision matrices under the bandable assumption on the Cholesky factor.

The rest of the paper is organized as follows. In section \ref{prelim}, we define our model, matrix norms, the parameter class and the decision theoretic prior selection. 
The convergence rates for precision matrices under the spectral norm and matrix $\ell_\infty$ norm are shown in section \ref{prec_main}. 
In section \ref{kchoice}, the practical choice of the bandwidth is proposed, and we conduct a simulation study in section \ref{prec_simul}. 
Discussion is given in section \ref{prec_disc}, and the proofs of the main results are in section \ref{proofsec}.

\section{Preliminaries}\label{prelim}
\subsection{Norms and Notations}
For any constants $a$ and $b$, $a \vee b$  denotes the maximum of $a$ and $b$. 
For any positive sequences $a_n$ and $b_n$, we denote $a_n =o(b_n)$ if $a_n/b_n \lra 0$ as $n\to \infty$. 
We denote $a_n \asymp b_n$ if there exist positive constants $C_1$ and $C_2$ such that $C_1 \le a_n/b_n \le C_2$ for all sufficiently large $n$, and  $a_n \lesssim b_n$ if there exists  a positive constant $C$ such that $a_n \le C b_n$ for all sufficiently large $n$.
For any $p\times p$ matrix $A$, $\lambda_{\min}(A)$ and $\lambda_{\max}(A)$ denote the minimum and maximum eigenvalue of the matrix $A$, respectively.

For any $p$-dimensional vector $a=(a_1,\ldots, a_p)\in\bbR^p$ and $p\times p$ matrix $A=(a_{ij})$, $B_{k,j}(a)$ and $B_k(A)$ are defined by  $B_{k,j}(a):=(b_i=a_i I(|i-j| \le k), 1\le i \le p)$ and $B_k(A) := (b_{ij}=a_{ij}I(|i-j|\le k), 1\le i,j \le p)$, respectively.

For any $p$-dimensional vector $a$, we define vector norms as follows: $\|a\|_1:= \sum_{i=1}^p |a_i|$, $\|a\|_2 := (\sum_{i=1}^pa_i^2)^{1/2}$ and $\|a\|_{\max}:=\max_{1\le i\le p}|a_i|$. With these norms, we define the operator norms for matrices. Let $A=(a_{ij})$ be a $p\times p$ matrix. The spectral norm (or matrix $\ell_2$ norm) is defined by
\bea
\|A\| &:=& \sup_{x\in\bbR^p \atop \|x\|_2=1} \|Ax \|_2 \,\,=\,\, (\lambda_{\max}(A^T A))^{1/2}.
\eea
We define the matrix $\ell_1$ norm, matrix $\ell_\infty$ norm and Frobenius norm by
\bea
\|A\|_1 &:=& \sup_{x\in\bbR^p \atop \|x\|_{1}=1} \|Ax\|_{1} \,\,=\,\, \max_{j} \sum_{i=1}^p |a_{ij}|, \\
\|A\|_\infty &:=& \sup_{x\in\bbR^p \atop \|x\|_{\max}=1} \|Ax\|_{\max} \,\,=\,\, \max_{i} \sum_{j=1}^p |a_{ij}|,\\
\|A\|_F &:=& \Big( \sum_{i=1}^p \sum_{j=1}^p a_{ij}^2 \Big)^{1/2},
\eea
respectively.
The max norm for matrices is defined by $\|A\|_{\max}:= \max_{i,j}|a_{ij}|$. 

\subsection{The Model and the Prior}
Suppose we observe a data set from the $p$-dimensional normal distribution
\bean\label{prec_model}
X_1,\ldots,X_n &\overset{iid}{\sim}& N_p(0, \sg_n),
\eean
where $\sg_n$ is a $p \times p$ positive definite matrix. We assume that $p = p_n$ is a function of $n$ increasing to $\infty$ as $n\to \infty$
and possibly $n = o(p)$. The unknown $p\times p$ true covariance matrix is denoted by $\sg_{0,n}$. 

For a $p\times p$ positive definite matrix $\Omega_n := \sg_n^{-1}$, the MCD guarantees that there uniquely exist a lower triangular matrix $A_n=(a_{jl})$ and a diagonal matrix $D_n=diag(d_j)$ such that
\bean\label{chol}
\Omega_n &=& (I_p - A_n)^T D_n^{-1} (I_p - A_n) ,
\eean
where $a_{jj} =0$ and $d_j >0$ for all $j=1,\ldots, p$. Note that the model \eqref{prec_model} with a precision matrix \eqref{chol} is equivalent to the following autoregressive model
\bea
X_{i1} &\overset{iid}{\sim}& N(0, d_1), \\
X_{ij} &\overset{ind}{\sim}& N \left( a_j^T Z_{ij} = \sum_{l=1}^{j-1}a_{jl}X_{il}, d_j \right) ,~ ~i,=1,\ldots,n,~j=2,\ldots,p ,
\eea
where $a_j:=(a_{j1},\ldots,a_{j,j-1})^T$ and $Z_{ij} := (X_{i1},\ldots, X_{i,j-1})^T$. 
Let $Y = (Y_1,\ldots, Y_p)^T \sim N_p(0, \sg)$, $Z_j = (Y_1,\ldots, Y_{j-1})^T$ and $\sg^{-1} = \Omega = (I_p -A)^T D^{-1}(I_p- A)$, then it is easy to check that, by the construction, the explicit forms of $a_j$ and $d_j$ are
\bea
a_j &=&  \V^{-1}(Z_j) \C(Z_j, Y_j) ,  \\
d_j &=& \V(Y_j) - \C(Y_j, Z_j) \V^{-1}(Z_j)\C(Z_j, Y_j) , ~j=2,\ldots,p,
\eea
where $d_1 = \V(Y_1)$. 
Bickel and Levina (2008b)\nocite{bickel2008regularized} approximated the precision matrix by considering only $k$ closest regressors $Z_{ij}^{(k)} := (X_{i,(j-k\vee 1)}, \ldots, X_{i,j-1})^T$ for each $X_{ij}$. It gives the new coefficient vector $a_j^{(k)} := \V^{-1}(Z_j^{(k)}) \C(Z_j^{(k)}, X_j)$. 
This is the same as assuming the lower triangle matrix $A_{n}$ in the MCD to be the $k$-banded lower triangular matrix. The resulting precision matrix 
$\Omega_{n} :=(I_p - A_{n})^T D_{n}^{-1}(I_p - A_{n})$ also becomes a $k$-banded matrix.

Bickel and Levina (2008b)\nocite{bickel2008regularized} suggested the ordinary least square estimators for $A_{n}$ and corresponding variance estimator $D_{n}$ under the $k$-banded assumption on $A_n$. Based on the least square estimators, they showed the convergence rates of covariance and precision matrix when $\log p = o(n)$.

In this paper, we suggest the following prior distribution
\begin{eqnarray}\label{kBCprior}
\begin{split}
\pi (a_{jl}) \,\,&\propto\,\, 1 ,~~ l=(j-k\vee 1),\ldots,j-1 ,\\
d_{j} \,\,&\overset{ind}{\sim}\,\, d_{j}^{\nu_0/2 -1} I(0< d_{j} < M) , ~~ j=1,\ldots,p ,
\end{split}
\end{eqnarray}
for some non-negative constants $M$ and $\nu_0$. We call the prior \eqref{kBCprior} the $k$-banded Cholesky ($k$-BC) prior.
The appropriate condition on $M$ and $\nu_0$ will be discussed in section \ref{prec_main}.

The prior \eqref{kBCprior} leads to the following joint posterior distribution,
\begin{eqnarray}\label{post}
\begin{split}
d_{j} \mid \bfX_n \,\,&\overset{ind}{\sim} \,\, IG^{Tr} \left( d_{j}\mid \frac{n_j}{2}, \frac{n}{2}  \what{d}_{jk}  ,  d_{j}\le M \right) , \quad j=1,\ldots, p \\
a_j \mid d_{j}, \bfX_n \,\,&\overset{ind}{\sim}\,\, N_{\min(j-1, k)} \left( a_j \mid  \what{a}_j^{(k)}   , ~\frac{d_{j}}{n}\what{\V}^{-1}(Z_j^{(k)})  \right) \quad j=2,\ldots,p,
\end{split}
\end{eqnarray}
where $\bfX_n := (X_1,\ldots,X_n)^T$, $n_j := n+ \nu_0- \min(j-1,k) -4$, $\what{d}_{1k}:=\what{\V}(X_1)$, $\what{\V}(X_j) := \frac{1}{n} \sum_{i=1}^n X_{ij}^2$, $\what{\C}(Z_j^{(k)}, X_j) := \frac{1}{n} \sum_{i=1}^n Z_{ij}^{(k)} X_{ij}$, $\what{\V}(Z_{j}^{(k)}) := \frac{1}{n} \sum_{i=1}^n Z_{ij}^{(k)} {Z_{ij}^{(k)}}^T$, 
\bea
\what{d}_{jk} &:=& \what{\V}(X_j)- \what{\C}(X_j, Z_j^{(k)}) \what{\V}^{-1}(Z_j^{(k)}) \what{\C}(Z_j^{(k)}, X_j) \quad\text{ and} \\
\what{a}_j^{(k)} &:=& \what{\V}^{-1}(Z_j^{(k)}) \what{\C}(Z_j^{(k)}, X_j),   ~~ j=2,\ldots,p.
\eea

Note that $IG(X \mid a,b)$ is the density function of the inverse-gamma random variable $X$ whose shape and rate parameters are $a$ and $b$, respectively. We denote $IG^{Tr}(X \mid a,b,A)$ as the truncated version of $IG(X\mid a,b)$ on support $A$.
$N_p(X \mid \mu, \sg)$ is the density function of the $p$-dimensional normal random variable $X$ whose mean vector and covariance matrix are $\mu$ and $\sg$, respectively.
The suggested prior on $d_{jk}$ has a compact support for a technical reason.

The zero-pattern of the Cholesky factor is related to the directed acyclic graph (DAG) (R{\"u}timann and B{\"u}hlmann, 2009\nocite{rutimann2009high}). The use of the $k$-BC prior \eqref{kBCprior} implies that we approximate the true model with a directed Gaussian graphical model. Thus, our method can be applied to directed Gaussian graphical models, but applications to graphical models will not be discussed in this paper.  
For more details about graphical models, see Lauritzen (1996)\nocite{lauritzen1996graphical}, Koller and Friedman (2009)\nocite{koller2009probabilistic} and R{\"u}timann and B{\"u}hlmann (2009).

\subsection{Parameter Class}
For a given constant $\epsilon_0 >0$ and a decreasing function $\gamma(k)\to 0$ as $k \lra \infty$, we define a class of precision matrices
\bea
\cU(\epsilon_0, \gamma) = \cU_p(\epsilon_0, \gamma)  &:=& \bigg\{ \Omega = (I_p - A)^T D^{-1}(I_p - A) \in \calC_p: \,\,\epsilon_0 \le \lambda_{\min}(\Omega)\le \lambda_{\max}(\Omega)\le \epsilon_0^{-1},\\
&&\quad\quad\quad\quad\quad  \|A - B_{k}(A)\|_\infty \le  \gamma(k),~\forall 0< k\le p-1  \bigg\} ,
\eea
where  $\calC_p$ is the class of all $p\times p$ dimensional positive definite matrices, and $A$ is a lower triangular matrix from the MCD of $\Omega$. 
Note that $\|A - B_{k}(A)\|_\infty \le  \gamma(k)$ is equivalent to $\max_{1\le i\le p} \sum_{j<i-k} |a_{ij}| \le  \gamma(k)$ where $A=(a_{ij})$.

We consider the following classes of $\gamma(k)$:
\begin{enumerate}
	\item (polynomially decreasing) $\gamma(k) = C k^{-\alpha}$ for some $\alpha>0$ and $C>0$;
	\item (exponentially decreasing) $\gamma(k) = C e^{-\beta k}$ for some $\beta>0$ and $C>0$; and
	\item (exact banding) for some $k_0>0$, $\gamma(k)=0$ for all $k > k_0$.
\end{enumerate}

Banerjee and Ghosal (2014) considered a similar parameter space for precision matrix defined by 
\bea
\cU^*(\epsilon_0, \gamma) = \cU_p^*(\epsilon_0, \gamma)  &:=& \bigg\{  \Omega= (\omega_{ij}) \in \calC_p  : 0<\epsilon_0 \le \lambda_{\min}(\Omega)\le \lambda_{\max}(\Omega)\le \epsilon_0^{-1},\\
&&\quad\quad\quad  \max_{1\le i\le p} \sum_{j:|i-j|>k} |\omega_{ij}| \le  \gamma(k),~\forall 0< k\le p-1 \bigg\}.
\eea

In fact, $\cU(\epsilon_0, \gamma)$ and $\cU^*(\epsilon_0,\gamma)$ are \emph{equivalent}, in terms of the convergence rate over them, if we consider an exponentially decreasing $\gamma (k)$ with $\beta > \log (\epsilon_0^{-2}+1)$ or an exact banding $\gamma (k)$. 
The following proposition describes the relation between them and its proof is given in Appendix.
\begin{proposition}\label{equi_para}
	Suppose $\gamma$  is a decreasing function defined on positive integers. If $\gamma$ is  exponentially decreasing  with  $\gamma(k)=Ce^{-\beta k}$ with $\beta > \log (\epsilon_0^{-2}+1)$ and $C>0$, or exact banding  for some $k_0 > 0$, then
	\bea
	\cU(\epsilon_0, C_1\gamma) \subseteq \cU^*(\epsilon_0,\gamma) \subseteq \cU(\epsilon_0, C_2\gamma)
	\eea
	for some positive constants $C_1$ and $C_2$ not depending on $p$.
\end{proposition}


\subsection{Bayesian Minimax Rate} 
Posterior convergence rate is one of the most commonly used measures to show the asymptotic concentration of posterior around the true parameter (Ghosal et al., 2000\nocite{ghosal2000convergence}; Ghosal and van der Vaart, 2007\nocite{ghosal2007posterior}). 
The concept of the posterior convergence rate is used to justify priors, but the best possible posterior convergence rate is an elusive concept to define. Motivated by the aforementioned difficulty,  Lee and Lee (2017) suggested a new decision theoretic framework for prior selection.

They considered a prior $\pi(\sg)$ as a decision rule and defined the P-loss as
\bea
\mathcal{L}(\sg_0 ,\pi) &:=& \bbE^\pi \left( d(\sg, \sg_0)\mid \bfX_n \right),
\eea
where $d(\sg,\sg')$ is a pseudometric on a set of positive definite matrices, $\sg_0$ is the true covariance matrix, and $\bbE^\pi (\cdot |\bfX_n)$ is the expectation under the posterior of $\Sigma$ when the prior $\pi$ and observation $\bfX_n$ are given. The P-risk is defined as
\bean\label{EPR}
\mathcal{R}(\sg_0, \pi) &:=& \bbE_{0} \bbE^\pi \left( d(\sg, \sg_0)\mid \bfX_n \right) ,
\eean
where $\bbE_{0} := \bbE_{\sg_{0}}$ denotes the expectation with respect to $\bfX_n \overset{iid}{\sim} N_p(0, \sg_{0})$.
Let $\cC_p$ be a class of $p\times p$ covariance matrices, and $\Pi_n$ be the class of all priors on $\cC_p$.  
Then, the Bayesian minimax rate of the posterior for the class $\cC^*_p\subset  \cC_p$ and the space of prior distributions $\Pi_n^*\subset \Pi_n$ is naturally defined as a sequence $r_n$ such that
\bea
\inf_{\pi\in \Pi_n^*} \sup_{\sg_0\in \cC_p^*} \bbE_{0} \mathcal{L}(\sg_0 ,\pi(\cdot | \bfX_n) ) &\asymp& r_n.
\eea
If a prior $\pi^*$ satisfies 
\bea
\sup_{\sg_0\in \cC_p^*} \bbE_{0} \mathcal{L}(\sg_0 ,\pi(\cdot | \bfX_n) ) &\lesssim& a_n,
\eea
then $\pi^*$ is said to have a P-loss convergence rate $a_n$, and if $a_n$ has the same rate with the Bayesian minimax rate, i.e. $a_n \asymp r_n$, $\pi^*$ is said to achieve the Bayesian minimax rate.
Thus, the use of the P-loss convergence rate enables to define the minimax rate of posterior clearer and makes the prior selection a mathematical problem. 
The P-loss convergence rate is a stronger measure than the posterior convergence rate, and the frequentist minimax lower bound is also a Bayesian minimax lower bound in general. 
For more details, see Lee and Lee (2017).

\section{Main Results}\label{prec_main}
\subsection{P-loss Convergence Rate and Bayesian Minimax Lower Bound  under Spectral Norm}\label{main_pre_spec_subsec}
In this subsection, we establish the Bayesian minimax lower and upper bounds under the spectral norm.
The P-loss convergence rate with the $k$-BC prior \eqref{kBCprior} is one of the main results of this paper. It is slightly slower than the rate of the Bayesian minimax lower bound given in Theorem \ref{main_pre_spec_LB}.  The proofs of theorems are given in section \ref{proofsec}.
Theorem \ref{main_pre_spec_LB} describes the frequentist minimax lower bound for precision matrices under the spectral norm.  
\begin{theorem}\label{main_pre_spec_LB}
	Consider model \eqref{prec_model} with  $p \le \exp(c n)$ for some constant $c>0$. Assume that $\Omega_n \in {\cal{U}}(\epsilon_0, \gamma)$ for given $\epsilon_0 > 0$ and a decreasing function $\gamma$. 
\benu
\item[(i)] If there exists a constant $k_0>0$ such that $\gamma(k)=0$ for all $k\ge k_0$, we have 
	\bea
	\inf_{\what{\Omega}_n } \sup_{\Omega_{0,n} \in \cU(\epsilon_0, \gamma)} \bbE_{0n}  \| \what{\Omega}_{n} - \Omega_{0,n}\|  &\gtrsim&  \left( \frac{\log p}{n} \right)^{1/2},
	\eea
	where $\what{\Omega}$ denotes an arbitrary estimator of $\Omega_{0,n}$.
\item[(ii)] If $\gamma(k) = Ce^{-\beta k}$ for some constants $\beta>0$ and $C>0$, then we have
	\bea
	\inf_{\what{\Omega}_n } \sup_{\Omega_{0,n} \in \cU(\epsilon_0, \gamma)} \bbE_{0n} \| \what{\Omega}_{n} - \Omega_{0,n}\|  &\gtrsim&  \min \left\{ \left( \frac{\log (n \vee p)}{n} \right)^{1/2},  \left(\,\frac{p}{n}\,\right)^{1/2} \right\}.
	\eea
\item[(iii)] 	 If $\gamma(k)=C k^{-\alpha}$ for some constants $\alpha>0$ and $C>0$, then we have
	\bea
	\inf_{\what{\Omega}_n } \sup_{\Omega_{0,n} \in \cU(\epsilon_0, \gamma)} \bbE_{0n}  \| \what{\Omega}_{n} - \Omega_{0,n}\|  &\gtrsim&  \min \left\{ \left( \frac{\log p}{n} \right)^{1/2} + n^{-\alpha/(2\alpha+1)} , \left(\,\frac{p}{n}\,\right)^{1/2} \right\}. 
	\eea 
\eenu
\end{theorem}

\begin{remark}
	 Since a frequentist minimax lower bound is also a P-loss minimax lower bound, Theorem \ref{main_pre_spec_LB} implies a P-loss minimax lower bound. For the proof of this argument, see Lee and Lee (2017)\nocite{lee2017optimal}. 
\end{remark}

To the best of our knowledge, there is no frequentist minimax lower bound result on this setting. The estimation of precision matrix with polynomially banded Cholesky factor under the spectral norm was studied by Bickel and Levina (2008b), but they did not consider the minimax lower bound. Verzelen (2010) obtained the minimax lower bound, but he considered the sparse Cholesky factor under the Frobenius norm.

Cai and Yuan (2016)  considered the estimation of covariance operator for random variables on a lattice graph under the spectral norm. They used both polynomially and exponentially banded assumption for the covariance operator.
Although bandable covariance (or precision) matrix classes and bandable Cholesky factor classes are different, 
if one considers one-dimensional lattice, interestingly the minimax lower bounds in Cai and Yuan (2016) coincide with the minimax lower bounds in Theorem \ref{main_pre_spec_LB} (ii) and (iii).

\begin{theorem}\label{main_pre_spec}
	Consider the model \eqref{prec_model} and the $k$-BC prior \eqref{kBCprior} with $M\ge 9\epsilon_0^{-1}$ and $\nu_0 = o(n)$. If $k^{3/2}(k + \log (n\vee p)) = O(n)$, $\sum_{m=1}^\infty \gamma(m) < \infty$ and $1\le k \le p$, 
	\bea
	\sup_{\Omega_{0,n} \in \cU(\epsilon_0, \gamma)} \bbE_{0n} \bbE^\pi \big( \|\Omega_{n}-\Omega_{0,n}\| \mid \bfX_n \big)  &\lesssim& k^{3/4} \left[ \left(\frac{k+\log (n\vee p)}{n} \right)^{1/2} + \gamma(k) \right].
	\eea 
\end{theorem}
Here, we use divide and conquer strategy to deal with the P-loss convergence rate. 
We decompose it into some small terms, which are easier to handle,
\bea
\bbE_{0n} \bbE^\pi \big( \|\Omega_{n}-\what{\Omega}_{nk}\| \mid \bfX_n \big) + \bbE_{0n} \|\what{\Omega}_{nk}-\Omega_{0,n}\| ,
\eea
where $\what{\Omega}_{nk}$ is a frequentist estimator of $\Omega_{0,n}$ with $k$-banded assumption. 
For the first term, we use concentration inequalities for posteriors of parameters around certain frequentist estimators. 
For the second term, some techniques for the frequentist convergence rate can be adopted. 
This strategy can be applied for general problems, for example, Castillo (2014)\nocite{castillo2014bayesian} also used the similar technique to obtain the P-loss convergence rate in density estimation.

When $\gamma(k)$ satisfies the exact banding with $k_0$, 
the prior \eqref{kBCprior} with $k\geq k_0$ not depending $n$ gives the P-loss convergence rate $(\log (n\vee p)/n)^{1/2}$, which is same as the Bayesian minimax rate when $p \geq n$. 
When $p < n$, if $p = n^\xi$ for some constant $0< \xi< 1$, the prior still achieves the Bayesian minimax rate. 

For the exponentially decreasing $\gamma(k)$, the optimal choice of $k$ is $(2\beta)^{-1}\log n$. It gives the P-loss convergence rate
\bean\label{up_sp2}
(\log n)^{3/4} \left(\frac{\log (n\vee p)}{n}\right)^{1/2} .
\eean
If $p \ge \log n$, the rate of \eqref{up_sp2} is same with the rate of minimax lower bound up to $(\log n)^{3/4}$. 

For the polynomially decreasing $\gamma(k)$, we assume that $p\ge n^{1/(2\alpha+1)}$.
The optimal choice of the bandwidth $k=\min \{ n^{1/(2\alpha+1)}, (n/\log p)^{1/(2\alpha)} \}$ gives the P-loss convergence rate
\bean\label{up_sp1}
 n^{-(4\alpha-3)/(8\alpha+4)} + \left(\frac{\log p}{n} \right)^{(4\alpha-3)/(8\alpha)}.
\eean
In other words, the P-loss convergence rate is $n^{-(4\alpha-3)/(8\alpha+4)}$ and $(\log p /n )^{(4\alpha-3)/(8\alpha)}$ when $p \le \exp(n^{1/(2\alpha+1)})$ and $p\ge \exp(n^{1/(2\alpha+1)})$, respectively.
Thus, if $p \ge n^{1/(2\alpha+1)}$, the P-loss convergence rate \eqref{up_sp1} is equal to the rate of the minimax lower bound up to the $\min(n^{3/(8\alpha+4)}, (n/\log p)^{3/(8\alpha)} )$ term.

\subsection{P-loss Convergence Rate and Bayesian Minimax Lower Bound under Matrix $\ell_{\infty}$ Norm}\label{main_pre_linf_subsec}
In this subsection, we establish the upper bound and lower bound for Bayesian minimax rate under the matrix $\ell_\infty$ norm.
The P-loss convergence rate with the $k$-BC prior \eqref{kBCprior} is one of the main results of this paper. It is slightly slower than the rate of the minimax lower bound given in Theorem \ref{main_pre_linf_LB}. However, we emphasize that our convergence rate is the fastest rate for bandable precision matrices among the existing Bayesian methods when we consider the exponentially decreasing or exact banding $\gamma(k)$. 
The proofs of theorems are given in section \ref{proofsec}.
Theorem \ref{main_pre_linf_LB} describes the minimax lower bound for precision matrices under the matrix $\ell_\infty$ norm.
\begin{theorem}\label{main_pre_linf_LB}
	Consider the model \eqref{prec_model} and let $p \le \exp(c n)$ for some constant $c>0$. 
	\benu
	\item[(i)] If there exists a constant $k_0$ such that $\gamma(k) =0$ for all $k\ge k_0$, we have 
	\bea
	\inf_{\what{\Omega}_n} \sup_{\Omega_{0,n} \in \cU(\epsilon_0, \gamma)} \bbE_{0n} \| \what{\Omega}_{n} - \Omega_{0,n}\|_\infty  &\gtrsim&   \left( \frac{ \log p  }{n} \right)^{1/2} .
	\eea
	\item[(ii)] If $\gamma(k) = Ce^{-\beta k}$ for some constants $\beta>0$ and $C>0$, then we have
	\bea
	\inf_{\what{\Omega}_n} \sup_{\Omega_{0,n} \in \cU(\epsilon_0, \gamma)} \bbE_{0n} \| \what{\Omega}_{n} - \Omega_{0,n}\|_\infty &\gtrsim& \min \left\{ \left( \frac{ \log p \cdot \log n}{n} \right)^{1/2}  , \frac{p}{\sqrt{n}}  \right\}.
	\eea
	\item[(iii)] If $\gamma(k)=Ck^{-\alpha}$ for some constants $\alpha>0$ and $C>0$, then we have
	\bea
	\inf_{\what{\Omega}_n} \sup_{\Omega_{0,n} \in \cU(\epsilon_0, \gamma)} \bbE_{0n} \| \what{\Omega}_{n} - \Omega_{0,n}\|_\infty &\gtrsim& \min \left\{ \left( \frac{ \log p}{n} \right)^{\alpha/(2\alpha+1)} + n^{-\alpha/(2\alpha+2)} , \frac{p}{\sqrt{n}}  \right\}.
	\eea
	\eenu 
\end{theorem}

\begin{theorem}\label{main_pre_linf}
	Consider the model \eqref{prec_model} and the $k$-BC prior \eqref{kBCprior} with $M\ge 9\epsilon_0^{-1}$ and $\nu_0 = o(n)$. If $k(k + \log (n\vee p)) = O(n)$, $\sum_{m=1}^\infty \gamma(m) < \infty$ and $1\le k \le p$, 
	\bea
	\sup_{\Omega_{0,n} \in \cU(\epsilon_0, \gamma)} \bbE_{0n} \bbE^\pi \big( \|\Omega_{n}-\Omega_{0,n}\|_\infty \mid \bfX_n \big)  &\lesssim& k \left[ \left(\frac{k+\log (n\vee p)}{n} \right)^{1/2} + \gamma(k)\right].
	\eea
\end{theorem}

\begin{remark}
	The P-loss convergence rate in Theorem \ref{main_pre_linf} is sharper than the posterior convergence rate of Banerjee and Ghosal (2014). 
	If we consider an exponentially decreasing or exact banding $\gamma(k)$, then the parameter spaces in two papers are equivalent by Proposition \ref{equi_para}. In that cases, the convergence rate obtained in Theorem \ref{main_pre_linf} is faster than that in  Banerjee and Ghosal (2014). 
\end{remark}


For the exact banding $\gamma(k)$ with $k_0$, the results are the same as those under the spectral norm. In words,  the prior \eqref{kBCprior} with $k\geq k_0$ gives P-loss convergence rate $(\log (n\vee p)/n)^{1/2}$, which is same as the optimal minimax rate when $p \geq n$. When $p < n$, if $p =  n^\xi$ for some constant $0 < \xi < 1$, the prior achieves the Bayesian minimax rate.

For the exponentially decreasing $\gamma(k)$, the optimal choice of $k$ is $(2\beta)^{-1}\log n$. It gives the P-loss convergence rate
\bean\label{up_inf2}
\log n \cdot\left(\frac{\log (n\vee p)}{n} \right)^{1/2} .
\eean
If $p \ge (\log n \cdot \log p)^{1/2}$, the rate of \eqref{up_inf2} is same with the rate of the minimax lower bound up to $(\log n \cdot \log(n\vee p) / \log p )^{1/2}$, which is $(\log n)^{1/2}$ provided that $p\ge n^\xi$ for some $\xi>0$.

For the polynomially decreasing $\gamma(k)$, we assume that $p \ge n^{1/(2\alpha+1)}$. 
The optimal choice of the bandwidth $k=\min \{ n^{1/(2\alpha+1)}, (n/\log p)^{1/(2\alpha)} \}$ gives the P-loss convergence rate
\bean\label{up_inf1}
n^{-(\alpha-1)/(2\alpha+1)} + \left(\frac{\log p}{n} \right)^{(\alpha-1)/(2\alpha)}.
\eean
In other words, the P-loss convergence rate is $ n^{-(\alpha-1)/(2\alpha+1)}$ and $(\log p/n)^{(\alpha-1)/(2\alpha)}$ when $p \le \exp(n^{1/(2\alpha+1)})$ and $p\ge \exp(n^{1/(2\alpha+1)})$, respectively.
Thus, if $n^{1/(2\alpha+2)} \le p\le \exp(n^{1/(2\alpha+1)})$, the P-loss convergence rate \eqref{up_inf1} is equal to the rate of the minimax lower bound up to the $n^{(\alpha+2)/(2(\alpha+1)(2\alpha+1))}$ term. If $p \ge \exp(n^{1/(2\alpha+1)})$, the P-loss convergence rate \eqref{up_inf1} is equal to the rate of the minimax lower bound up to the $(n/\log p)^{(\alpha+1)/(4\alpha^2+2\alpha)}$ term.

\subsection{The Frequentist Convergence Rates and the Posterior Convergence Rates}\label{freq_post}
In this subsection, we obtain the frequentist convergence rate and the traditional posterior convergence rate of the $k$-BC prior \eqref{kBCprior}. 
For the frequentist convergence rate, we propose a plug-in estimator,
\bea
\what{\Omega}_{nk}^{LL} &:=& (I_p- \bbE^{{\pi}}(A_{n} | \bfX_n))^T \bbE^{\tilde{\pi}}(D_{n}^{-1}| \bfX_n) (I_p- \bbE^{{\pi}}(A_{n}| \bfX_n)),
\eea
where $\bbE^{\tilde{\pi}}(\cdot \mid \bfX_n)$ are posterior means using the nontruncated posteriors,
\bea
\tilde{\pi}(d_{j}\mid \bfX_n) &=& IG\left(d_{j}\mid \frac{n_j}{2}, \frac{n}{2}\what{d}_{jk}  \right), \quad j=1,\ldots,p.
\eea
The plug-in estimator $\what{\Omega}_{nk}^{LL}$ is more convenient than the posterior mean $\bbE^\pi(\Omega_{n} | \bfX_n)$ in practice because of its simple form. Note that $\bbE^{{\pi}}(a_j\mid d_{j}, \bfX_n) = \what{a}_j^{(k)}$ and $\bbE^{\tilde{\pi}}(d_{j}^{-1}\mid \bfX_n) = n_j \what{d}_{jk}^{-1}/n$.
As a justification for the use of nontruncated posterior mean, in Corollary \ref{pluginmean}, we show that $\what{\Omega}_{nk}^{LL}$ achieves the same rate with the P-loss convergence rate.
The proof of Corollary \ref{pluginmean} is given in Appendix \ref{proofsec}.

According to Proposition 2.1 of Lee and Lee (2017), the P-loss convergence rate is a posterior convergence rate. 
Thus, the rates obtained in  Theorem \ref{main_pre_spec} and Theorem \ref{main_pre_linf} in this paper are also  the posterior convergence rates.

\begin{corollary}\label{pluginmean}
	Consider the model \eqref{prec_model}, and assume $\sum_{m=1}^\infty \gamma(m) < \infty$, $\nu_0=o(n)$ and $1\le k \le p$. If $k^{3/2} (k+\log (n\vee p)) = O(n)$,
	\bea
	\sup_{\Omega_{0,n} \in \cU(\epsilon_0, \gamma)}  \bbE_{0n} \|\what{\Omega}_{nk}^{LL} -\Omega_{0,n}\|   &\lesssim& k^{3/4} \left[ \left(\frac{k+\log (n\vee p)}{n} \right)^{1/2} +  \gamma(k)\right] . 
	\eea
	If $k(k+\log (n\vee p)) = O(n)$, 
	\bea
	\sup_{\Omega_{0,n} \in \cU(\epsilon_0, \gamma)}  \bbE_{0n} \|\what{\Omega}_{nk}^{LL} -\Omega_{0,n}\|_\infty    &\lesssim& k \left[\left(\frac{k+\log (n\vee p)}{n} \right)^{1/2} +  \gamma(k)\right] .
	\eea
\end{corollary}

\begin{corollary}\label{coro_postconv}
	Consider the model \eqref{prec_model} and the $k$-BC prior \eqref{kBCprior} with $M\ge 9\epsilon_0^{-1}$ and $\nu_0 = o(n)$. 
	Assume $\sum_{m=1}^\infty \gamma(m) < \infty$ and $1\le k \le p$.
	If $k^{3/2}(k+\log (n\vee p)) = O(n)$ and $\epsilon_n = k^{3/4} \left[(k+\log (n\vee p))/n )^{1/2} + \gamma(k)\right]$, then for any $M_n \to \infty$ as $n\to\infty$,
	\bea
	\sup_{\Omega_{0,n} \in \cU(\epsilon_0, \gamma)} \bbE_{0n} \Big[ \pi \big( \|\Omega_{n}-\Omega_{0,n}\| \ge M_n \epsilon_n \mid \bfX_n \big) \Big]  &\lra& 0. 
	\eea
	If $k(k+\log (n\vee p)) = O(n)$ and $\epsilon_n^* = k \big[\left((k+\log (n\vee p))/n \right)^{1/2} + \gamma(k)\big]$, then for any $M_n \to \infty$ as $n\to\infty$,
	\bea
	\sup_{\Omega_{0,n} \in \cU(\epsilon_0, \gamma)} \bbE_{0n} \Big[ \pi \big( \|\Omega_{n}-\Omega_{0,n}\|_\infty \ge M_n \epsilon_n^* \mid \bfX_n \big) \Big]  &\lra& 0.
	\eea
\end{corollary}

\section{Choice of the Bandwidth $k$}\label{kchoice}
In this subsection, we suggest using the posterior mode of $k$ as a practical choice of the bandwidth $k$.
Using Theorem \ref{main_pre_spec} and Theorem \ref{main_pre_linf}, one can calculate the optimal {\emph{rate}} of the bandwidth $k$ minimizing the P-loss convergence rate, when the rate of $\gamma(k)$ is given. 
However,  in practice $\gamma(k)$ is not known and $k$ can not be chosen based on $\gamma(k)$.

Let $\pi(k)$ be a prior distribution for the bandwidth $k$ and $f(\bfX_n\mid A_{n}, D_{n},k)$ be the likelihood function based on the observation $\bfX_n$.
In section \ref{prec_simul}, the prior distribution of $k$ was set by $\pi(k)\propto  \exp(-k^4)$ as in Banerjee and Ghosal (2014).
The marginal posterior for $k$ is easily derived as
\begin{align}
\pi(k \mid \bfX_n) \,\,&\propto\,\, \pi(k) \int \int f(\bfX_n \mid A_{n}, D_{n}, k) \pi( A_{n}, D_{n}) d A_{n} d D_{n}  \nonumber \\
\begin{split}\label{k_post}
&\propto\,\, \pi(k)\prod_{j=2}^p \det\left(n\what{\V}(Z_j^{(k)})/(2\pi) \right)^{-1/2} \Gamma\left(\frac{n_j}{2}\right) \left(\frac{n}{2}\what{d}_{jk} \right)^{-n_j/2} \\
& \times \,\,\,\, \prod_{j=1}^p F_{IG} \left( M  \,\,\Big|\,\, n_j/2, n\what{d}_{jk}/2 \right)
\end{split}
\end{align}
by routine calculations where $F_{IG}(M \mid a,b)$ is a distribution function of $IG(a,b)$. Since the marginal posterior \eqref{k_post} has a simple analytic form, the posterior mode, say $\hat{k}$, can be easily obtained.
The performance of $\hat{k}$ is described through comparisons with other approaches in the next section. 

Note that the Cholesky-based Bayes estimator $\what{\Omega}_{nk}^{LL}$ is similar to the banded estimator (Bickel and Levina, 2008b)\nocite{bickel2008regularized}, $\what{\Omega}_{nk}^{BL}$. The major difference between two estimators is the choice of the bandwidth parameter $k$. 
It is worthwhile to compare the practical performances of the two schemes for selecting the bandwidth $k$. 
Bickel and Levina (2008b)\nocite{bickel2008regularized} proposed a resampling scheme to estimate the oracle $k$. To estimate the minimizer of the risk
\bean\label{BLrisk}
R(k) &=& \bbE_{0n}\| \what{\Omega}_{nk}^{BL} -\Omega_{0,n} \|_1 ,
\eean
they divided $n$ observations into two groups of sizes $n_1=n/3$ and $n_2 = n-n_1$, randomly. 
We computed the banded estimator $\what{\Omega}_{1,nk}^{BL}$ using the first group as an estimator for $\what{\Omega}_{nk}^{BL}$. 
Since the sample precision matrix is computationally unstable for large $p$, the banded estimator $\what{\Omega}_{2,nK}^{BL}$ was used instead of the sample precision matrix for the second group as an estimator for $\Omega_{0,n}$. Here, $K = \min(n,p)-1$, but in the simulation study in this paper, we used $K=20$ to reduce the computation time.
In the same way, $t$-th random split gives $\what{\Omega}_{1,nk}^{BL,(t)}$ and $\what{\Omega}_{2,nK}^{BL,(t)}$ for $t=1,\ldots,T$. The risk \eqref{BLrisk} was approximated by 
\bean\label{BLriskapp}
\what{R}(k) &=& \frac{1}{T} \sum_{t=1}^T \| \what{\Omega}_{1,nk}^{BL,(t)} - \what{\Omega}_{2,nK}^{BL,(t)}\|_1,
\eean
and the bandwidth $k$ was selected as $\hat{k}^{BL}= \argmin_k \what{R}(k)$. 
For more detailed description about the resampling scheme, see Bickel and Levina (2008b)\nocite{bickel2008regularized}. 

\section{Simulation Study}\label{prec_simul}
We investigated the performance of the proposed Bayes estimator $\what{\Omega}_{nk}^{LL}$ and the posterior mode $\hat{k}$. The performances of the Bayes estimator based on the $G$-Wishart prior $\what{\Omega}_{nk}^{BG}$ (Banerjee and Ghosal, 2014), the banded estimator $\what{\Omega}_{nk}^{BL}$ (Bickel and Levina, 2008b) and the graphical maximum likelihood estimator (MLE) $\what{\Omega}_{nk}^{MLE}$ (Lauritzen, 1996)\nocite{lauritzen1996graphical} were compared in various scenarios. For the proposed estimator $\what{\Omega}_{nk}^{LL}$, we used $\nu_0 = 2$ throughout this section.

Banerjee and Ghosal (2014) proposed two Bayes estimators corresponding to the Stein's loss and the squared-error loss, respectively. 
We checked the performances of two Bayes estimators, say $\what{\Omega}_{nk}^{BG1}$ and $\what{\Omega}_{nk}^{BG2}$, with $\delta=3$. 
For these estimators, the bandwidth $k$ was chosen by the posterior mode in Banerjee and Ghosal (2014), $\hat{k}^{BG}$.
We examined the performances of two banded estimator (Bickel and Levina, 2008b),  $\what{\Omega}_{nk_1}^{BL}$ and  $\what{\Omega}_{nk_2}^{BL}$, where the banding parameter $k$ of the former was chosen by $\hat{k}^{BL}$ and that of the latter was chosen by $\hat{k}$.
For the graphical MLE, the bandwidth $k$ was  set by $\hat{k}$.

The spectral norm, matrix $\ell_{\infty}$ norm and Frobenius norm were used as the loss functions. 
The sample sizes $n=100, 200$ and $500$ and the dimensions $p=100, 200$ and $500$ were investigated. For each settings, the values of the loss function,
\bean\label{loss_sim}
\| \what{\Omega}_{nk}^{(s)} - \Omega_{0,n} \| ,\quad s=1,\ldots,100,
\eean
were calculated with $100$ simulated data for each methods $\what{\Omega}_{nk}$ and loss functions $\|\cdot\|$ where $\Omega_{0,n}$ denotes the true precision matrix. The mean and standard deviation of \eqref{loss_sim} were used as summary statistics.
We considered the following true precision matrices.

\begin{example}($AR(1)$ process)
		Assume the true covariance matrix $\sg = (\sigma_{ij})$ is given by
		\bea
		\sigma_{ij} &=& \rho^{|i-j|}, \,\, 1\le i,j \le p
		\eea
		with $\rho =0.3$. Then the true precision matrix is a banded matrix with $AR(1)$ process structure. 
\end{example}

\begin{example}($AR(4)$ process)
	Assume the true precision matrix $\Omega= (\omega_{ij})$ is given by
	\bea
	\omega_{ij} &=& I(|i-j|=0) + 0.4 \cdot I(|i-j|=1) + 0.2\cdot I(|i-j|=2) \\
	&& \,\,+\,\, 0.2 \cdot I(|i-j|=3) + 0.1 \cdot I(|i-j|=4).
	\eea
	Thus, the true precision matrix is a banded matrix with the $AR(4)$ process structure. Furthermore, it is always positive definite because of the diagonally dominant property.  
\end{example}

\begin{example}(Long-range dependence)
	The last example deals with the situation where the true precision matrix is not a bandable matrix in $\cU(\epsilon_0,\gamma)$. 
	Consider a fractional Gaussian noise model and the true covariance matrix $\sg = (\sigma_{ij})$ is given by
	\bea
	\sigma_{ij} &=& \frac{1}{2} \left( ||i-j|+1 |^{2H} - 2|i-j|^{2H} + | |i-j|-1|^{2H}  \right), \,\, 1\le i,j \le p
	\eea
	with $H \in [0.5, 1]$. The Hurst parameter $H$ indicates the dependency of the process. 
	$H=0.5$ implies the white noise, while $H$ near 1 means the long-range dependence. 
	We chose $H=0.7$. In this case, the true precision matrix does not belong to the bandable class.
\end{example}

\begin{table}
	\scriptsize\centering
	\caption{Simulation results for $AR(1)$ model. For each $n$ and $p$, the mean and standard deviation (in parenthesis) of three loss functions (the spectral norm, matrix $\ell_\infty$ norm and Frobenius norm) were calculated.}
	\begin{tabular}{lll ccccc}
		\hline
		\multicolumn{3}{c}{} & LL & BG & BL1 & BL2 & MLE \\ \hline
		\multicolumn{1}{c}{\multirow{18}{*}{$n=100$}} & \multicolumn{1}{c}{\multirow{6}{*}{$p=100$}} & \multicolumn{1}{c}{\multirow{2}{*}{$\|\cdot\|$\,\,\,\,\,}} & 0.720	&	0.759	&	1.217	&	0.786	&	0.786 \\
		& & & (0.139)	&	(0.141)	&	(0.391)	&	(0.146)	&	(0.146) \\
		& & \multicolumn{1}{c}{\multirow{2}{*}{$\|\cdot\|_\infty$}} & 0.913	&	0.957	&	1.905	&	0.989	&	0.989 \\
		& & & (0.176)	&	(0.177)	&	(0.813)	&	(0.184)	&	(0.184) \\
		& & \multicolumn{1}{c}{\multirow{2}{*}{$\|\cdot\|_F$}} & 2.382	&	2.447	&	3.837	&	2.503	&	2.503 \\
		& & & (0.171)	&	(0.187)	&	(1.067)	&	(0.196)	&	(0.196) \\ \cline{2-8}
		& \multicolumn{1}{c}{\multirow{6}{*}{$p=200$}} & \multicolumn{1}{c}{\multirow{2}{*}{$\|\cdot\|$\,\,\,\,\,}} & 0.802	&	0.842	&	1.294	&	0.873	&	0.873 \\
		& & & (0.140)	&	(0.140)	&	(0.353)	&	(0.145)	&	(0.145) \\
		& & \multicolumn{1}{c}{\multirow{2}{*}{$\|\cdot\|_\infty$}} & 1.025	&	1.071	&	2.044	&	1.108	&	1.108 \\
		& & & (0.180)	&	(0.180)	&	(0.716)	&	(0.186)	&	(0.186) \\
		& & \multicolumn{1}{c}{\multirow{2}{*}{$\|\cdot\|_F$}} & 3.395	&	3.487	&	5.471	&	3.567	&	3.567 \\
		& & & (0.165)	&	(0.179)	&	(0.322)	&	(0.188)	&	(0.188) \\ \cline{2-8}
		& \multicolumn{1}{c}{\multirow{6}{*}{$p=500$}} & \multicolumn{1}{c}{\multirow{2}{*}{$\|\cdot\|$\,\,\,\,\,}} & 0.910	&	0.951	&	1.504	&	0.985	&	0.985 \\
		& & & (0.147)	&	(0.146)	&	(0.417)	&	(0.152)	&	(0.152) \\
		& & \multicolumn{1}{c}{\multirow{2}{*}{$\|\cdot\|_\infty$}} & 1.151	&	1.196	&	2.412	&	1.239	&	1.239 \\
		& & & (0.181)	&	(0.181)	&	(0.928)	&	(0.188)	&	(0.188) \\
		& & \multicolumn{1}{c}{\multirow{2}{*}{$\|\cdot\|_F$}} & 5.377	&	5.521	&	9.070	&	5.647	&	5.647 \\
		& & & (0.172)	&	(0.186)	&	(2.353)	&	(0.353)	&	(0.195) \\ \hline
		\multicolumn{1}{c}{\multirow{18}{*}{$n=200$}} & \multicolumn{1}{c}{\multirow{6}{*}{$p=100$}} & \multicolumn{1}{c}{\multirow{2}{*}{$\|\cdot\|$\,\,\,\,\,}} & 0.482	&	0.498	&	0.585	&	0.507	&	0.507 \\
		& & & (0.090)	&	(0.094)	&	(0.165)	&	(0.096)	&	(0.096) \\
		& & \multicolumn{1}{c}{\multirow{2}{*}{$\|\cdot\|_\infty$}} & 0.619	&	0.636	&	0.815	&	0.646	&	0.646 \\
		& & & (0.117)	&	(0.121)	&	(0.315)	&	(0.124)	&	(0.124) \\
		& & \multicolumn{1}{c}{\multirow{2}{*}{$\|\cdot\|_F$}} & 1.673	&	1.696	&	1.991	&	1.714	&	1.714 \\
		& & & (0.110)	&	(0.116)	&	(0.442)	&	(0.119)	&	(0.119) \\ \cline{2-8}
		& \multicolumn{1}{c}{\multirow{6}{*}{$p=200$}} & \multicolumn{1}{c}{\multirow{2}{*}{$\|\cdot\|$\,\,\,\,\,}} & 0.537	&	0.556	&	0.644	&	0.567	&	0.567 \\
		& & & (0.098)	&	(0.100)	&	(0.154)	&	(0.102)	&	(0.102) \\
		& & \multicolumn{1}{c}{\multirow{2}{*}{$\|\cdot\|_\infty$}} & 0.685	&	0.706	&	0.896	&	0.718	&	0.718 \\
		& & & (0.127)	&	(0.130)	&	(0.277)	&	(0.133)	&	(0.133) \\
		& & \multicolumn{1}{c}{\multirow{2}{*}{$\|\cdot\|_F$}} & 2.374	&	2.406	&	2.851	&	2.432	&	2.432 \\
		& & & (0.113)	&	(0.121)	&	(0.544)	&	(0.124)	&	(0.124) \\ \cline{2-8}
		& \multicolumn{1}{c}{\multirow{6}{*}{$p=500$}} & \multicolumn{1}{c}{\multirow{2}{*}{$\|\cdot\|$\,\,\,\,\,}} & 0.594	&	0.615	&	0.747	&	0.626	&	0.626 \\
		& & & (0.080)	&	(0.080)	&	(0.156)	&	(0.082)	&	(0.082) \\
		& & \multicolumn{1}{c}{\multirow{2}{*}{$\|\cdot\|_\infty$}} & 0.755	&	0.777	&	1.054	&	0.792	&	0.792 \\
		& & & (0.108)	&	(0.109)	&	(0.326)	&	(0.111)	&	(0.111) \\
		& & \multicolumn{1}{c}{\multirow{2}{*}{$\|\cdot\|_F$}} & 3.762	&	3.813	&	4.692	&	3.855	&	3.855 \\
		& & & (0.104)	&	(0.111)	&	(0.866)	&	(0.114)	&	(0.114) \\ \hline
		\multicolumn{1}{c}{\multirow{18}{*}{$n=500$}} & \multicolumn{1}{c}{\multirow{6}{*}{$p=100$}} & \multicolumn{1}{c}{\multirow{2}{*}{$\|\cdot\|$\,\,\,\,\,}} & 0.287	&	0.292	&	0.309	&	0.295	&	0.295 \\
		& & & (0.045)	&	(0.046)	&	(0.055)	&	(0.047)	&	(0.047) \\
		& & \multicolumn{1}{c}{\multirow{2}{*}{$\|\cdot\|_\infty$}} & 0.368	&	0.373	&	0.404	&	0.376	&	0.376 \\
		& & & (0.060)	&	(0.063)	&	(0.084)	&	(0.064)	&	(0.064) \\
		& & \multicolumn{1}{c}{\multirow{2}{*}{$\|\cdot\|_F$}} & 1.053	&	1.060	&	1.110	&	1.064	&	1.064 \\
		& & & (0.065)	&	(0.067)	&	(0.110)	&	(0.067)	&	(0.067) \\ \cline{2-8}
		& \multicolumn{1}{c}{\multirow{6}{*}{$p=200$}} & \multicolumn{1}{c}{\multirow{2}{*}{$\|\cdot\|$\,\,\,\,\,}} & 0.314	&	0.321	&	0.340	&	0.324	&	0.324 \\
		& & & (0.045)	&	(0.046)	&	(0.065)	&	(0.047)	&	(0.047) \\
		& & \multicolumn{1}{c}{\multirow{2}{*}{$\|\cdot\|_\infty$}} & 0.405	&	0.412	&	0.445	&	0.415	&	0.415 \\
		& & & (0.059)	&	(0.061)	&	(0.101)	&	(0.062)	&	(0.062) \\
		& & \multicolumn{1}{c}{\multirow{2}{*}{$\|\cdot\|_F$}} & 1.489	&	1.497	&	1.565	&	1.503	&	1.503 \\
		& & & (0.073)	&	(0.074)	&	(0.172)	&	(0.074)	&	(0.074) \\ \cline{2-8}
		& \multicolumn{1}{c}{\multirow{6}{*}{$p=500$}} & \multicolumn{1}{c}{\multirow{2}{*}{$\|\cdot\|$\,\,\,\,\,}} & 0.340	&	0.347	&	0.359	&	0.350	&	0.350 \\
		& & & (0.042)	&	(0.042)	&	(0.051)	&	(0.043)	&	(0.043) \\
		& & \multicolumn{1}{c}{\multirow{2}{*}{$\|\cdot\|_\infty$}} & 0.436	&	0.444	&	0.465	&	0.448	&	0.448 \\
		& & & (0.053)	&	(0.053)	&	(0.078)	&	(0.053)	&	(0.053) \\
		& & \multicolumn{1}{c}{\multirow{2}{*}{$\|\cdot\|_F$}} & 2.352	&	2.365	&	2.433	&	2.375	&	2.375 \\
		& & & (0.069)	&	(0.070)	&	(0.188)	&	(0.071)	&	(0.071) \\ \hline
	\end{tabular}\label{table:sim1}
\end{table}

\begin{table}
	\scriptsize\centering
	\caption{Simulation results for $AR(4)$ model. For each $n$ and $p$, the mean and standard deviation (in parenthesis) of three loss functions (the spectral norm, matrix $\ell_\infty$ norm and Frobenius norm) were calculated.}
	\begin{tabular}{lll ccccc}
		\hline
		\multicolumn{3}{c}{} & LL & BG & BL1 & BL2 & MLE \\ \hline
		\multicolumn{1}{c}{\multirow{18}{*}{$n=100$}} & \multicolumn{1}{c}{\multirow{6}{*}{$p=100$}} & \multicolumn{1}{c}{\multirow{2}{*}{$\|\cdot\|$\,\,\,\,\,}} & 1.510	&	1.475	&	1.481	&	1.473	&	1.473 \\
		& & & (0.040)	&	(0.041)	&	(0.340)	&	(0.041)	&	(0.041) \\
		& & \multicolumn{1}{c}{\multirow{2}{*}{$\|\cdot\|_\infty$}} & 1.854	&	1.826	&	2.446	&	1.827	&	1.827 \\
		& & & (0.058)	&	(0.061)	&	(0.607)	&	(0.063)	&	(0.063) \\
		& & \multicolumn{1}{c}{\multirow{2}{*}{$\|\cdot\|_F$}} & 5.130	&	5.050	&	4.189	&	5.046	&	5.046 \\
		& & & (0.070)	&	(0.065)	&	(0.620)	&	(0.065)	&	(0.065) \\ \cline{2-8}
		& \multicolumn{1}{c}{\multirow{6}{*}{$p=200$}} & \multicolumn{1}{c}{\multirow{2}{*}{$\|\cdot\|$\,\,\,\,\,}} & 1.541	&	1.506	&	1.668	&	1.504	&	1.504 \\
		& & & (0.034)	&	(0.035)	&	(0.395)	&	(0.035)	&	(0.035) \\
		& & \multicolumn{1}{c}{\multirow{2}{*}{$\|\cdot\|_\infty$}} & 1.899	&	1.873	&	2.873	&	1.874	&	1.874 \\
		& & & (0.065)	&	(0.069)	&	(0.678)	&	(0.071)	&	(0.071) \\
		& & \multicolumn{1}{c}{\multirow{2}{*}{$\|\cdot\|_F$}} & 7.312	&	7.196	&	6.015	&	7.191	&	7.191 \\
		& & & (0.072)	&	(0.068)	&	(0.840)	&	(0.068)	&	(0.068) \\ \cline{2-8}
		& \multicolumn{1}{c}{\multirow{6}{*}{$p=500$}} & \multicolumn{1}{c}{\multirow{2}{*}{$\|\cdot\|$\,\,\,\,\,}} & 1.564	&	1.530	&	1.884	&	1.528	&	1.528 \\
		& & & (0.029)	&	(0.030)	&	(0.368)	&	(0.030)	&	(0.030) \\
		& & \multicolumn{1}{c}{\multirow{2}{*}{$\|\cdot\|_\infty$}} & 1.938	&	1.913	&	3.061	&	1.915	&	1.915 \\
		& & & (0.052)	&	(0.056)	&	(0.654)	&	(0.057)	&	(0.057) \\
		& & \multicolumn{1}{c}{\multirow{2}{*}{$\|\cdot\|_F$}} & 11.610	&	11.426	&	9.620	&	11.417	&	11.417 \\
		& & & (0.076)	&	(0.072)	&	(1.288)	&	(0.072)	&	(0.072) \\ \hline
		\multicolumn{1}{c}{\multirow{18}{*}{$n=200$}} & \multicolumn{1}{c}{\multirow{6}{*}{$p=100$}} & \multicolumn{1}{c}{\multirow{2}{*}{$\|\cdot\|$\,\,\,\,\,}} & 1.313	&	1.461	&	0.843	&	1.288	&	1.288 \\
		& & & (0.314)	&	(0.027)	&	(0.168)	&	(0.325)	&	(0.325) \\
		& & \multicolumn{1}{c}{\multirow{2}{*}{$\|\cdot\|_\infty$}} & 1.616	&	1.734	&	1.366	&	1.596	&	1.596 \\
		& & & (0.273)	&	(0.050)	&	(0.284)	&	(0.280)	&	(0.280) \\
		& & \multicolumn{1}{c}{\multirow{2}{*}{$\|\cdot\|_F$}} & 4.477	&	4.949	&	2.513	&	4.431	&	4.431 \\
		& & & (0.980)	&	(0.049)	&	(0.260)	&	(0.972)	&	(0.972) \\ \cline{2-8}
		& \multicolumn{1}{c}{\multirow{6}{*}{$p=200$}} & \multicolumn{1}{c}{\multirow{2}{*}{$\|\cdot\|$\,\,\,\,\,}} & 0.972	&	1.482	&	0.482	&	0.934	&	0.934 \\
		& & & (0.289)	&	(0.027)	&	(0.171)	&	(0.299)	&	(0.299) \\
		& & \multicolumn{1}{c}{\multirow{2}{*}{$\|\cdot\|_\infty$}} & 1.343	&	1.759	&	1.457	&	1.319	&	1.319 \\
		& & & (0.245)	&	(0.043)	&	(0.280)	&	(0.249)	&	(0.249) \\
		& & \multicolumn{1}{c}{\multirow{2}{*}{$\|\cdot\|_F$}} & 4.574	&	7.047	&	3.528	&	4.502	&	4.502 \\
		& & & (1.357)	&	(0.054)	&	(0.336)	&	(1.356)	&	(1.356) \\ \cline{2-8}
		& \multicolumn{1}{c}{\multirow{6}{*}{$p=500$}} & \multicolumn{1}{c}{\multirow{2}{*}{$\|\cdot\|$\,\,\,\,\,}} & 0.871	&	1.499	&	1.015	&	0.840	&	0.840 \\
		& & & (0.052)	&	(0.023)	&	(0.169)	&	(0.059)	&	(0.059) \\
		& & \multicolumn{1}{c}{\multirow{2}{*}{$\|\cdot\|_\infty$}} & 1.300	&	1.800	&	1.643	&	1.291	&	1.291 \\
		& & & (0.097)	&	(0.041)	&	(0.303)	&	(0.117)	&	(0.117) \\
		& & \multicolumn{1}{c}{\multirow{2}{*}{$\|\cdot\|_F$}} & 6.083	&	11.200	&	5.686	&	6.001	&	6.001 \\
		& & & (0.345)	&	(0.058)	&	(0.554)	&	(0.226)	&	(0.226) \\ \hline
		\multicolumn{1}{c}{\multirow{18}{*}{$n=500$}} & \multicolumn{1}{c}{\multirow{6}{*}{$p=100$}} & \multicolumn{1}{c}{\multirow{2}{*}{$\|\cdot\|$\,\,\,\,\,}} & 0.501	&	1.052	&	0.450	&	0.513	&	0.513 \\
		& & & (0.139)	&	(0.395)	&	(0.084)	&	(0.122)	&	(0.122) \\
		& & \multicolumn{1}{c}{\multirow{2}{*}{$\|\cdot\|_\infty$}} & 0.767	&	1.281	&	0.733	&	0.784	&	0.784 \\
		& & & (0.151)	&	(0.355)	&	(0.144)	&	(0.137)	&	(0.137) \\
		& & \multicolumn{1}{c}{\multirow{2}{*}{$\|\cdot\|_F$}} & 1.663	&	3.573	&	1.439	&	1.676	&	1.676 \\
		& & & (0.444)	&	(1.324)	&	(0.118)	&	(0.411)	&	(0.411) \\ \cline{2-8}
		& \multicolumn{1}{c}{\multirow{6}{*}{$p=200$}} & \multicolumn{1}{c}{\multirow{2}{*}{$\|\cdot\|$\,\,\,\,\,}} & 0.447	&	0.807	&	0.481	&	0.473	&	0.473 \\
		& & & (0.063)	&	(0.255)	&	(0.067)	&	(0.067)	&	(0.067) \\
		& & \multicolumn{1}{c}{\multirow{2}{*}{$\|\cdot\|_\infty$}} & 0.722	&	1.083	&	0.768	&	0.754	&	0.754 \\
		& & & (0.080)	&	(0.229)	&	(0.094)	&	(0.088)	&	(0.088) \\
		& & \multicolumn{1}{c}{\multirow{2}{*}{$\|\cdot\|_F$}} & 1.939	&	3.764	&	2.010	&	1.985	&	1.985 \\
		& & & (0.068)	&	(1.245)	&	(0.107)	&	(0.075)	&	(0.075) \\ \cline{2-8}
		& \multicolumn{1}{c}{\multirow{6}{*}{$p=500$}} & \multicolumn{1}{c}{\multirow{2}{*}{$\|\cdot\|$\,\,\,\,\,}} & 0.493	&	0.737	&	0.530	&	0.522	&	0.522 \\
		& & & (0.081)	&	(0.044)	&	(0.084)	&	(0.085)	&	(0.085) \\
		& & \multicolumn{1}{c}{\multirow{2}{*}{$\|\cdot\|_\infty$}} & 0.784	&	1.036	&	0.835	&	0.820	&	0.820 \\
		& & & (0.111)	&	(0.056)	&	(0.117)	&	(0.116)	&	(0.116) \\
		& & \multicolumn{1}{c}{\multirow{2}{*}{$\|\cdot\|_F$}} & 3.069	&	5.151	&	3.189	&	3.139	&	3.139 \\
		& & & (0.067)	&	(0.456)	&	(0.141)	&	(0.076)	&	(0.076) \\ \hline
	\end{tabular}\label{table:sim2}
\end{table}

\begin{table}
	\scriptsize\centering
	\caption{Simulation results for fractional Gaussian noise model. For each $n$ and $p$, the mean and standard deviation (in parenthesis) of three loss functions (the spectral norm, matrix $\ell_\infty$ norm and Frobenius norm) were calculated.}
	\begin{tabular}{lll ccccc}
		\hline
		\multicolumn{3}{c}{} & LL & BG & BL1 & BL2 & MLE \\ \hline
		\multicolumn{1}{c}{\multirow{18}{*}{$n=100$}} & \multicolumn{1}{c}{\multirow{6}{*}{$p=100$}} & \multicolumn{1}{c}{\multirow{2}{*}{$\|\cdot\|$\,\,\,\,\,}} & 0.837  &	0.880  &	1.232  &	0.911  &	0.911 \\
		& & & (0.149)  &	(0.149)  &	(0.357)  &	(0.154)  &	(0.154) \\
		& & \multicolumn{1}{c}{\multirow{2}{*}{$\|\cdot\|_\infty$}} & 1.588  &	1.636  &	2.284  &	1.676  &	1.676 \\
		& & & (0.194)  &	(0.193)  &	(0.698)  &	(0.200)  &	(0.200) \\
		& & \multicolumn{1}{c}{\multirow{2}{*}{$\|\cdot\|_F$}} & 2.879  &	2.955  &	3.980  &	3.030  &	3.030 \\
		& & & (0.194)  &	(0.211)  &	(0.908)  &	(0.222)  &	(0.222) \\ \cline{2-8}
		& \multicolumn{1}{c}{\multirow{6}{*}{$p=200$}} & \multicolumn{1}{c}{\multirow{2}{*}{$\|\cdot\|$\,\,\,\,\,}} & 0.931  &	0.973  &	1.326  &	1.008  &	1.008 \\
		& & & (0.147)  &	(0.146)  &	(0.335)  &	(0.152)  &	(0.152) \\
		& & \multicolumn{1}{c}{\multirow{2}{*}{$\|\cdot\|_\infty$}} & 1.752  &	1.800  &	2.495  &	1.843  &	1.843 \\
		& & & (0.188)  &	(0.187)  &	(0.654)  &	(0.194)  &	(0.194) \\
		& & \multicolumn{1}{c}{\multirow{2}{*}{$\|\cdot\|_F$}} & 4.100  &	4.208  &	5.806  &	4.315  &	4.315 \\
		& & & (0.178)  &	(0.194)  &	(1.236)  &	(0.204)  &	(0.204) \\ \cline{2-8}
		& \multicolumn{1}{c}{\multirow{6}{*}{$p=500$}} & \multicolumn{1}{c}{\multirow{2}{*}{$\|\cdot\|$\,\,\,\,\,}} & 1.043  &	1.085  &	1.493  &	1.124  &	1.124 \\
		& & & (0.156)  &	(0.155)  &	(0.348)  &	(0.161)  &	(0.161) \\
		& & \multicolumn{1}{c}{\multirow{2}{*}{$\|\cdot\|_\infty$}} & 1.929  &	1.976  &	2.800  &	2.025  &	2.025 \\
		& & & (0.196)  &	(0.194)  &	(0.740)  &	(0.202)  &	(0.202) \\
		& & \multicolumn{1}{c}{\multirow{2}{*}{$\|\cdot\|_F$}} & 6.478  &	6.648  &	9.321  &	6.817  &	6.817 \\
		& & & (0.185)  &	(0.202)  &	(0.945)  &	(0.212)  &	(0.212) \\ \hline
		\multicolumn{1}{c}{\multirow{18}{*}{$n=200$}} & \multicolumn{1}{c}{\multirow{6}{*}{$p=100$}} & \multicolumn{1}{c}{\multirow{2}{*}{$\|\cdot\|$\,\,\,\,\,}} & 0.601  &	0.622  &	0.659  &	0.634  &	0.634 \\
		& & & (0.096)  &	(0.096)  &	(0.126)  &	(0.098)  &	(0.098) \\
		& & \multicolumn{1}{c}{\multirow{2}{*}{$\|\cdot\|_\infty$}} & 1.269  &	1.293  &	1.360  &	1.307  &	1.307 \\
		& & & (0.137)  &	(0.138)  &	(0.217)  &	(0.140)  &	(0.140) \\
		& & \multicolumn{1}{c}{\multirow{2}{*}{$\|\cdot\|_F$}} & 2.287  &	2.318  &	2.435  &	2.347  &	2.347 \\
		& & & (0.123)  &	(0.131)  &	(0.278)  &	(0.134)  &	(0.134) \\ \cline{2-8}
		& \multicolumn{1}{c}{\multirow{6}{*}{$p=200$}} & \multicolumn{1}{c}{\multirow{2}{*}{$\|\cdot\|$\,\,\,\,\,}} & 0.665  &	0.687  &	0.719  &	0.699  &	0.699 \\
		& & & (0.111)  &	(0.111)  &	(0.116)  &	(0.113)  &	(0.113) \\
		& & \multicolumn{1}{c}{\multirow{2}{*}{$\|\cdot\|_\infty$}} & 1.400  &	1.424  &	1.484  &	1.440  &	1.440 \\
		& & & (0.148)  &	(0.148)  &	(0.175)  &	(0.150)  &	(0.150) \\
		& & \multicolumn{1}{c}{\multirow{2}{*}{$\|\cdot\|_F$}} & 3.244  &	3.289  &	3.465  &	3.330  &	3.330 \\
		& & & (0.131)  &	(0.139)  &	(0.289)  &	(0.143)  &	(0.143) \\ \cline{2-8}
		& \multicolumn{1}{c}{\multirow{6}{*}{$p=500$}} & \multicolumn{1}{c}{\multirow{2}{*}{$\|\cdot\|$\,\,\,\,\,}} & 0.733  &	0.755  &	0.801  &	0.769  &	0.769 \\
		& & & (0.092)  &	(0.092)  &	(0.104)  &	(0.094)  &	(0.094) \\
		& & \multicolumn{1}{c}{\multirow{2}{*}{$\|\cdot\|_\infty$}} & 1.526  &	1.551  &	1.646  &	1.568  &	1.568 \\
		& & & (0.127)  &	(0.127)  &	(0.181)  &	(0.129)  &	(0.129) \\
		& & \multicolumn{1}{c}{\multirow{2}{*}{$\|\cdot\|_F$}} & 5.141  &	5.212  &	5.547  &	5.277  &	5.277 \\
		& & & (0.117)  &	(0.124)  &	(0.464)  &	(0.128)  &	(0.128) \\ \hline
		\multicolumn{1}{c}{\multirow{18}{*}{$n=500$}} & \multicolumn{1}{c}{\multirow{6}{*}{$p=100$}} & \multicolumn{1}{c}{\multirow{2}{*}{$\|\cdot\|$\,\,\,\,\,}} & 0.406  &	0.434  &	0.436  &	0.420  &	0.420 \\
		& & & (0.047)  &	(0.043)  &	(0.046)  &	(0.049)  &	(0.049) \\
		& & \multicolumn{1}{c}{\multirow{2}{*}{$\|\cdot\|_\infty$}} & 1.004  &	1.004  &	1.046  &	1.020  &	1.020 \\
		& & & (0.076)  &	(0.073)  &	(0.074)  &	(0.077)  &	(0.077) \\
		& & \multicolumn{1}{c}{\multirow{2}{*}{$\|\cdot\|_F$}} & 1.722  &	1.875  &	1.869  &	1.742  &	1.742 \\
		& & & (0.064)  &	(0.066)  &	(0.084)  &	(0.066)  &	(0.066) \\ \cline{2-8}
		& \multicolumn{1}{c}{\multirow{6}{*}{$p=200$}} & \multicolumn{1}{c}{\multirow{2}{*}{$\|\cdot\|$\,\,\,\,\,}} & 0.433  &	0.467  &	0.468  &	0.448  &	0.448 \\
		& & & (0.045)  &	(0.044)  &	(0.046)  &	(0.046)  &	(0.046) \\
		& & \multicolumn{1}{c}{\multirow{2}{*}{$\|\cdot\|_\infty$}} & 1.086  &	1.128  &	1.130  &	1.105  &	1.105 \\
		& & & (0.073)  &	(0.073)  &	(0.074)  &	(0.074)  &	(0.074) \\
		& & \multicolumn{1}{c}{\multirow{2}{*}{$\|\cdot\|_F$}} & 2.433  &	2.649  &	2.647  &	2.460  &	2.460 \\
		& & & (0.063)  &	(0.065)  &	(0.085)  &	(0.065)  &	(0.065) \\ \cline{2-8}
		& \multicolumn{1}{c}{\multirow{6}{*}{$p=500$}} & \multicolumn{1}{c}{\multirow{2}{*}{$\|\cdot\|$\,\,\,\,\,}} & 0.461  &	0.495  &	0.498  &	0.476  &	0.476 \\
		& & & (0.043)  &	(0.039)  &	(0.040)  &	(0.044)  &	(0.044) \\
		& & \multicolumn{1}{c}{\multirow{2}{*}{$\|\cdot\|_\infty$}} & 1.164  &	1.203  &	1.209  &	1.183  &	1.183 \\
		& & & (0.073)  &	(0.062)  &	(0.061)  &	(0.074)  &	(0.074) \\
		& & \multicolumn{1}{c}{\multirow{2}{*}{$\|\cdot\|_F$}} & 3.856  &	4.189  &	4.202  &	3.900  &	3.900 \\
		& & & (0.062)  &	(0.071)  &	(0.085)  &	(0.064)  &	(0.064) \\ \hline
	\end{tabular}\label{table:sim3}
\end{table}

\begin{figure*}[!b]
	\centering
	\includegraphics[width=16cm,height=7cm]{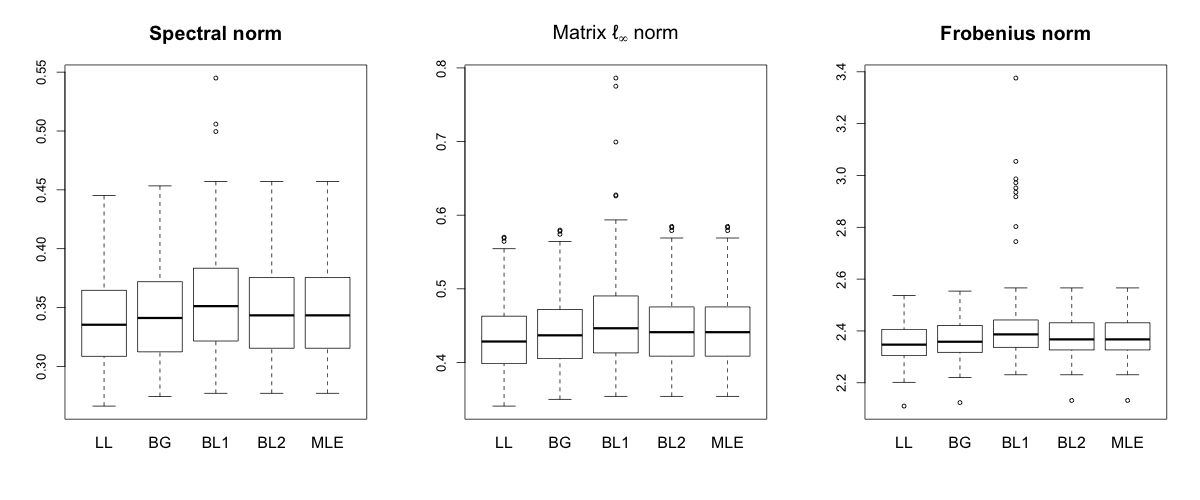}
	\vspace{-2cm}\caption{The average errors for $AR(4)$ process structure precision matrix under the spectral norm, matrix $\ell_{\infty}$ norm and Frobenius norm. The sample size $n$ and the dimensionality $p$ are 500.}
	\label{fig:boxplot}
\end{figure*}
Table \ref{table:sim1}--\ref{table:sim3} show the simulation results for the above three examples, and Figure \ref{fig:boxplot} shows the performance of each estimator when the true precision matrix is $AR(4)$ process and $(n,p)=(500,500)$.
We omitted the estimator $\what{\Omega}_{nk}^{BG2}$ because its performance is quite similar to that of $\what{\Omega}_{nk}^{BG1}$ throughout the all scenarios. $BG$ in table \ref{table:sim1}--\ref{table:sim3} and Figure \ref{fig:boxplot} represents $\what{\Omega}_{nk}^{BG1}$.
There are two major remarks on the simulation results.
First, it seems that the proposed Bayes estimator is practically comparable or better than the method of Banerjee and Ghosal (2014).
Since theoretical results in this paper and in Banerjee and Ghosal (2014) are based on the optimal choice of $k$ depending some unknown parameters, the practical performances using the posterior modes $\hat{k}$ are of independent interest. 
Especially when the sample size is large, the performance of $\what{\Omega}_{nk}^{BL}$ is often better than $\what{\Omega}_{nk}^{BG1}$.
Second, our selection scheme for $k$ is comparable to that of Bickel and Levina (2008b).
The $BL1$ columns and $BL2$ columns in tables show the results for banded estimators with $k$ chosen by $\hat{k}^{BL}$ and $\hat{k}$, respectively.
They show similar performance in our simulation study.

\section{Discussion}\label{prec_disc}
We suggested the $k$-BC prior \eqref{kBCprior} for bandable precision matrices via the MCD. 
The P-loss convergence rates for precision matrices under the spectral norm and matrix $\ell_\infty$ norm were established. 
Although the P-loss convergence rates are slightly slower than the rate of the Bayesian minimax lower bounds, the proposed approach attains faster posterior convergence rate than those of the other existing Bayesian methods.
Simulation study supported that its practical performance is comparable to those of other competitive approaches.


There are a few possible extensions of this paper related to the bandwidth $k$. 
Firstly, theoretical results in this paper depend on the unknown parameter of $\gamma(k)$. To choose the optimal $k$, one should know about the rate of $\gamma(k)$.
Thus, developing an adaptive procedure, which simultaneously attains a reasonable convergence rate regardless of $\gamma(k)$, is one of the possible extension. 
Secondly, the theoretical property of the posterior mode $\hat{k}$ is unexplored. A theoretical result similar to the Theorem 4 in Bickel and Levina (2008a) can be investigated.

\section{Proofs}\label{proofsec}

\subsection{Proof of Proposition \ref{equi_para}}
\begin{proof}
		We only prove the exponentially decreasing case, $\gamma(k)= Ce^{-\beta k}$ for some $\beta>0$ and $C>0$, because the proposition is trivially hold for the exact banding case. 
		
		Suppose $\Omega \in \cU(\epsilon_0,\gamma)$ and let $\Omega= (I_p -A)^T D^{-1}(I_p -A)$ where $A=(a_{ij})$ and $D=diag(d_j)$. 
		One can check that $\|D^{-1}\| \le \epsilon_0^{-1}$ and
		\bea
		\|A\|_{\max} &\le& \max_j \|a_j \|_2 \\
		&\le& \max_j \|\V(Z_j)^{-1}\| \| \V(Z_{j+1})\| \,\,\le\,\, \epsilon_0^{-2}.
		\eea
		
		Furthermore,
		\bean
		\|A- B_k(A)\|_\infty = \max_{i} \sum_{j<i-k}|a_{ij}| &\le& C\gamma(k) ,\label{Ainfband_assmp} \\
		\|A- B_k(A)\|_1 = \max_j \sum_{i>j+k} |a_{ij}| &\le& \sum_{m=k}^\infty \gamma(m) \,\,\le\,\, C'\gamma(k) , \label{Al1band_assmp}
		\eean
		for some $C'>1$ because $\gamma(k) = Ce^{-\beta k}$. 
		Note that $\omega_{pp} = d_p^{-1}$ and
		\bean
		\begin{split}\label{omega_ij}
			\omega_{ij} &=\,\, -d_j^{-1} a_{ji} + \sum_{l=j+1}^p d_l^{-1} a_{li}a_{lj} \quad\text{ for any $1\le i<j \le  p $.}
		\end{split}
		\eean
		
		Then for $1\le i<p$, define $k$ so that $i= p-k-1$. Then, $k \geq 0$ and
		\bea
		|\omega_{ip}| &=& d_p^{-1} |a_{pi}| \\
		&\le& C \epsilon_0^{-1}\gamma(k)
		\eea
		by \eqref{Ainfband_assmp}.
		On the other hand, for $1\le i<j \le p-1$, define $k$ so that $j-i=k+1$. Then, $k \geq 0$ and 
		\bea
		|\omega_{ij}| &=& | -d_ja_{ji} + \sum_{l=j+1}^p d_l^{-1}a_{li}a_{lj}| \\
		&\le& d_j^{-1}|a_{ji}| + \sum_{l=j+1}^p d_l^{-1}|a_{li} a_{lj}| \\
		&\le& \epsilon_0^{-3}\left(|a_{ji}| +  \sum_{l=j+1}^p |a_{li}|\right) \\
		&=& \epsilon_0^{-3} \sum_{l=j}^p |a_{li}| \,\,\le\,\, \epsilon_0^{-3} C' \gamma(k)
		\eea
		
		by \eqref{Al1band_assmp}. Thus, we have
		\bea
		\| \Omega - B_k(\Omega)\|_\infty &=& \max_i \sum_{j:|i-j|>k} |\omega_{ij}|\\
		&\le& \max_i \sum_{j>i+k}|\omega_{ij}| + \max_i \sum_{j<i-k} |\omega_{ji}|  \\
		&\le& 2\epsilon_0^{-3}C' \sum_{m=k}^\infty \gamma(m)   \,\,\le\,\, C''\gamma(k)
		\eea
		for some constant $C''>0$. This proves the first inequality.

		Suppose $\Omega \in \cU^*(\epsilon_0, \gamma)$. We need to prove that $\max_{i} \sum_{j<i-k}|a_{ij}| = \max_{i} \sum_{j=1}^{i-k-1}|a_{ij}| \le C\gamma(k)$ for some constant $C>0$. Note that from \eqref{omega_ij}, we have 
		\bean\label{t=0}
		d_p^{-1}\sum_{j=1}^{p-k-1} |a_{pj}| &=& \sum_{j=1}^{p-k-1}|\omega_{jp}|   \,\,\le\,\,  \gamma(k),
		\eean
		for any $0\le k \le p-2$. 
		We will show that 
		\bean\label{asmaller_gamma}
		d_{p-t}^{-1} \sum_{j=1}^{p-t-k-1} |a_{p-t,j}| &\le& \gamma(k) + \epsilon_0^{-2}\sum_{m=1}^t (1+\epsilon_0^{-2})^{m-1} \gamma(k+m)
		\eean
		for any $1\le t \le p-k-2$ for some $0\le k \le p-3$. 
		Then, \eqref{t=0} and \eqref{asmaller_gamma} imply $\Omega \in \cU(\epsilon_0,C\gamma)$ for some $C>0$ because $\max_{j}d_j \le \max_j \V(X_j) \le \epsilon_0^{-1}$ and we assume that $\gamma(k) = e^{-\beta k}$ and $\beta > \log (\epsilon_0^{-2}+1)$.

		By \eqref{omega_ij} and the assumption $\Omega \in \cU^*(\epsilon_0,\gamma)$,
		\bean\label{pminus1}
		\sum_{j=1}^{p-k-2}|-d_{p-1}^{-1}a_{p-1,j} + d_p^{-1}a_{pj}a_{p,p-1}| &=&  \sum_{j=1}^{p-k-2}|\omega_{j,p-1}|  \,\,\le\,\,\gamma(k)
		\eean
		for any $0\le k\le p-3$.
		Thus, \eqref{t=0} and \eqref{pminus1} imply that
		\bea
		d_{p-1}^{-1}\sum_{j=1}^{p-k-2}|a_{p-1,j}| &\le& \sum_{j=1}^{p-k-2}|-d_{p-1}^{-1}a_{p-1,j} + d_p^{-1}a_{pj}a_{p,p-1}| +  \sum_{j=1}^{p-k-2} |d_p^{-1}a_{pj}a_{p,p-1}|  \\
		&\le& \gamma(k) + \epsilon_0^{-2}\gamma(k+1) 
		\eea
		because $\Omega \in \cU^*(\epsilon_0, \gamma)$ means $|a_{p,p-1}|\le \|A\|_{\max} \le \epsilon_0^{-2}$. Thus, \eqref{asmaller_gamma} holds for $t=1$. 
		Now assume that \eqref{asmaller_gamma} holds for $t-1$ and consider for the case of $t$.
		Note that
		\bea
		\gamma(k)&\ge& \sum_{j=1}^{p-t-k-1}|\omega_{j,p-t}| \\
		&=& \sum_{j=1}^{p-t-k-1} \left| -d_{p-t}^{-1}a_{p-t,j} + \sum_{l=p-t+1}^p d_l^{-1}a_{lj}a_{l,p-t}  \right| ,
		\eea
		which implies that
		\begin{align}
		&d_{p-t}^{-1} \sum_{j=1}^{p-t-k-1} |a_{p-t,j}| \nonumber\\
		&\le \gamma(k) + \sum_{j=1}^{p-t-k-1} \sum_{l=p-t+1}^p d_l^{-1}|a_{lj}a_{l,p-t}| \nonumber \\
		&\le \gamma(k) + \epsilon_0^{-2}  \sum_{l=p-t+1}^p d_l^{-1}\sum_{j=1}^{p-t-k-1} |a_{lj}| \nonumber\\
		&= \gamma(k) + \epsilon_0^{-2} \sum_{t_1=0}^{t-1} d_{p-t_1}^{-1}\sum_{j=1}^{p-t_1-(k+ t-t_1)-1} |a_{p-t_1,j}| \nonumber\\
		\begin{split}\label{induc_gammak}
		&\le \gamma(k) + \epsilon_0^{-2}\gamma(k+t) \\
		&+ \epsilon_0^{-2} \sum_{t_1=1}^{t-1} \left( \gamma(k+t-t_1) + \epsilon_0^{-2}\sum_{m=1}^{t_1}(1+\epsilon_0^{-2})^{m-1}\gamma(k+t-t_1+m)  \right).
		\end{split}
		\end{align}
		In \eqref{induc_gammak}, one can check that the coefficient of $\gamma(k+t-t')$ is
		\bea
		\epsilon_0^{-2}+ \epsilon_0^{-4} \sum_{m=1}^{t-t' -1}(1+\epsilon_0^{-2})^{m-1} &=& \epsilon_0^{-2}(1+\epsilon_0^{-2})^{t- t'-1}
		\eea
		for $0\le t3' \le t-1$, and the coefficient of $\gamma(k)$ is $1$. Thus,
		\bea
		d_{p-t}^{-1} \sum_{j=1}^{p-t-k-1} |a_{p-t,j}| &\le& \gamma(k) + \epsilon_0^{-2} \sum_{m=1}^{t} (1+\epsilon_0^{-2})^{m-1} \gamma(k+m). 
		\eea
		This completes the proof by induction. \quad $\blacksquare$
\end{proof}

\subsection{Proof of the Minimax Lower Bounds: Theorem \ref{main_pre_spec_LB} and Theorem \ref{main_pre_linf_LB}}
\begin{proof}[Proof of Theorem \ref{main_pre_spec_LB}.]
	We follow closely the line of a proof in Cai et al. (2010)\nocite{cai2010optimal}. 
	Consider the polynomially decreasing case, $\gamma(k)= C k^{-\alpha}$, first. 
	Two parameter classes are considered depending on the relation between $p$ and $n$. 
	For $\exp(n^{1/(2\alpha+1)}) \ge p$ case, we show that
	\bean\label{spec_LB_1}
	\inf_{\what{\Omega}} \sup_{\Omega\in \cU_{11}} \bbE_{0n} \|\what{\Omega} - \Omega \|&\gtrsim& \min \left( n^{-\alpha/(2\alpha+1)}, \sqrt{\frac{p}{n}}\, \right) ,
	\eean
	and for $\exp(n^{1/(2\alpha+1)}) \le p$ case, we show that
	\bean\label{spec_LB_2}
	\inf_{\what{\Omega}} \sup_{\Omega\in \cU_{12}} \bbE_{0n} \|\what{\Omega} - \Omega \|&\gtrsim&  \left(\frac{\log p}{n}\right)^{1/2}
	\eean
	for some $\cU_{11} \cup \cU_{12} \subset \cU(\epsilon_0, \gamma)$. 
	
	Consider $\exp(n^{1/(2\alpha+1)}) \ge p$ case first. Without loss of generality, we assume $k= \min (n^{1/(2\alpha+1)}, p)$ is an even number, and define a class of precision matrices
	\bea
	\cU_{11} &:=& \bigg\{\Omega(\theta) \in \bbR^{p\times p} : \Omega(\theta)= (I_p - A(\theta))^T (I_p- A(\theta)), \\
	&& \quad\quad A(\theta) =-\tau a \sum_{m=1}^{k/2}\theta_m B(m,k), \theta=(\theta_m, 1\leq m \leq k/2)\in \{0,1\}^{k/2}   \bigg\}
	\eea
	where $B(m,k) := (b_{ij} = I(i=m+1,\ldots, k \text{ and } j=m), 1\leq i, j \leq p)$ is a $p\times p$ matrix and $a:=(nk)^{-1/2}$. If we choose sufficiently small constant $\tau>0$, it is easy to check that for any $\Omega(\theta)\in \cU_{11}$, $\epsilon_0 \le \lambda_{\min}(\Omega(\theta))\le \lambda_{\max}(\Omega(\theta))\le \epsilon_0^{-1}$ and $\|A(\theta) - B_{k_1}(A(\theta))\|_\infty \le C {k_1}^{-\alpha} $ for any $k_1 >0$, 
	so that $\cU_{11}\subset \cU(\epsilon_0, \gamma)$ for all sufficiently large $n$.
	
	We use the Assouad's lemma
	\bea
	\inf_{\what{\Omega}} \sup_{\Omega(\theta)\in \cU_{11}} 2 \bbE_{\theta} \| \what{\Omega} - \Omega(\theta)\| &\ge& \min_{H(\theta,\theta')\ge 1} \frac{\|\Omega(\theta)-\Omega(\theta')\|}{H(\theta,\theta')} \cdot \frac{k/2}{2} \cdot \min_{H(\theta,\theta')=1} \| \bbP_\theta \wedge \bbP_{\theta'} \|
	\eea
	where $H(\theta,\theta') := \sum_{m=1}^{k/2} |\theta_m-\theta_m'|, \|\bbP_\theta \wedge \bbP_{\theta'} \| := \int p_{\theta}\wedge p_{\theta'} d\mu$ and $p_\theta$ is a density function of observation $\bfX_n$ which follows  $N(0, \Omega(\theta)^{-1})$. 
	If we show that
	\bean\label{Assouad_omegaH}
	\min_{H(\theta,\theta')\ge 1} \frac{\|\Omega(\theta)-\Omega(\theta')\|}{H(\theta,\theta')} &\gtrsim& a
	\eean
	and
	\bean\label{Assouad_PP}
	\min_{H(\theta,\theta')=1} \| \bbP_\theta \wedge \bbP_{\theta'} \| &\ge& c
	\eean
	for some constant $c>0$, it will complete the proof.	
	To show \eqref{Assouad_omegaH}, define a $p$-dimensional vector $v:=(I(k/2 \le i\le k), 1\leq i \leq p)$. Then,
	\bea
	\|\Omega(\theta)-\Omega(\theta')\| &\ge& \frac{\|(\Omega(\theta)- \Omega(\theta'))v\|_2}{\|v\|_2} \\
	&\ge& \frac{\sqrt{(k/2\cdot \tau a)^2 H(\theta,\theta')}}{\sqrt{k/2}} \\
	&=& \sqrt{\frac{k/2}{H(\theta,\theta')}} \cdot \tau a H(\theta,\theta') \\
	&\ge& \tau a H(\theta,\theta').
	\eea
	Thus, we have shown the first part. To show \eqref{Assouad_PP}, note that
	\bea
	\| \bbP_\theta \wedge \bbP_{\theta'} \| &=& 1 - \frac{1}{2} \|\bbP_\theta  - \bbP_{\theta'}\|_1.
	\eea
	
	Thus, it suffices to show that $\|\bbP_\theta  - \bbP_{\theta'}\|_1^2 \le 1/2$. Also note that
	\bea
	\|\bbP_\theta  - \bbP_{\theta'}\|_1 &\le& 2 K(\bbP_{\theta'} \mid \bbP_\theta ) \\
	&=& n \left[ tr(\Omega(\theta')^{-1}\Omega(\theta)) - \log \det (\Omega(\theta')^{-1}\Omega(\theta)) -p  \right] \\
	&=& n \left[tr(\Omega(\theta')^{-1} D_1) - \log \det (\Omega(\theta')^{-1} D_1 + I_p) \right] \\
	&=& n \left[tr(\Omega(\theta')^{-1/2} D_1 \Omega(\theta')^{-1/2}) - \log \det (\Omega(\theta')^{-1/2} D_1 \Omega(\theta')^{-1/2} + I_p) \right]
	\eea
	where $K(\bbP_{\theta'} \mid \bbP_\theta ) := \int \log(\frac{dP_{\theta'}}{dP_\theta}) dP_{\theta'}$ is the Kullback-Leibler divergence and $D_1 := \Omega(\theta) - \Omega(\theta')$. 
	Let $\Omega(\theta')^{-1} = U VU^T$ be the diagonalization of $\Omega(\theta')^{-1}$. $U$ is a orthogonal matrix whose columns are the eigenvectors of $\Omega(\theta')^{-1}$, and $V$ is a diagonal matrix whose $i$th diagonal element is the eigenvalue of $\Omega(\theta')^{-1}$ corresponding to the $i$th column of $U$. It is easy to check that
	\bea
	\| \Omega(\theta')^{-1/2} D_1 \Omega(\theta')^{-1/2}\|_F^2 
	&=& \|UV^{1/2}U^T D_1 UV^{1/2}U^T\|_F^2 \\
	&=& \|V^{1/2}U^T D_1 UV^{1/2}\|_F^2 \\
	&\le& \|V\|^2 \|U^TD_1 U\|_F^2 \\
	&=& \|\Omega(\theta')^{-1}\|^2 \|D_1\|_F^2 \\
	&\le& C k (\tau a)^2
	\eea
	for some constant $C>0$. 
	Since $\Omega(\theta')^{-1/2} D_1 \Omega(\theta')^{-1/2} + I_p$ is a positive definite matrix and $\| \Omega(\theta')^{-1/2} D_1 \Omega(\theta')^{-1/2}\|_F^2$ is small,  
	\bea
	\|\bbP_\theta  - \bbP_{\theta'}\|_1 &\le& n R_n
	\eea
	where $R_n \le C\|\Omega(\theta')^{-1/2} D_1 \Omega(\theta')^{-1/2}\|_F^2$ for some constant $C>0$,  by Lemma A.7 in Lee and Lee (2016). 	
	Thus, we have $\|\bbP_\theta -\bbP_{\theta'}\|_1 \le 1/2$ for some small $\tau>0$ because $nk a^2 =1$.
	
	Now consider $\exp(n^{1/(2\alpha+1)}) \le p$ case.
	To show \eqref{spec_LB_2}, define a class of diagonal precision matrices
	\bea
	\cU_{12} &:=& \left\{\Omega_m \in \bbR^{p\times p} : \Omega_m = I_p +\tau \left(\frac{\log p}{n}\right)^{1/2} \Big(I(i=j=m) \Big),~0\le m\le p  \right\}
	\eea
	for some small $\tau>0$. Since $p \le \exp(cn)$ for some constant $c>0$, $\cU_{12}\subset \cU(\epsilon_0, \gamma)$ holds trivially. Let $r_{\min} := \inf_{1\le m\le p} \|\Omega_0 - \Omega_m \|$. 
	We use the Le Cam's lemma (Le Cam, 1973)\nocite{lecam1973convergence} 
	\bea
	\inf_{\what{\Omega}} \sup_{\Omega_m \in \cU_{12}} \bbE_m \|\what{\Omega}-\Omega_m\| &\ge& \frac{1}{2}\cdot r_{\min} \cdot \| \bbP_0 \wedge \bar{\bbP} \|
	\eea
	where $\bar{\bbP} := p^{-1} \sum_{m=1}^{p} \bbP_m$ and $\bbP_m$ is the distribution function of $N(0, \Omega_m^{-1})$ with observation $\bfX_n$. Note that $r_{\min} = \tau (\log p/n)^{1/2}$.
	We only need to show that $\| \bbP_0 \wedge \bar{\bbP} \| \ge c$ for some constant $c>0$. 
	By the same argument with Cai et al. (2010, page 2129)\nocite{cai2010optimal}, it suffices to show that 
	\bean\label{mixden_zero}
	\int \frac{(p^{-1} \sum_{m=1}^{p}f_m)^2}{f_0} d\mu -1 &\lra& 0,
	\eean
		as $n\to\infty$ where $f_m$ is the density function of $\bbP_m$ with respect to a $\sigma$-finite measure $\mu$. 
	Note that
	\bea
	\int \frac{(p^{-1} \sum_{m=1}^{p}f_m)^2}{f_0} d\mu -1 
	&=& \frac{1}{p^2} \sum_{m=1}^{p} \int \frac{f_m^2}{f_0}d\mu + \frac{1}{p^2}\sum_{m\neq j} \int \frac{f_m f_j}{f_0} d\mu -1
	\eea
	and $\int f_m f_j/f_0 d\mu =1$ for any $m\neq j$. 
		Also note that 
	\bea
	\int \frac{f_m^2}{f_0} d\mu &=& (1+b)^{n/2} \left(1- \frac{b}{1+2b} \right)^{n/2} \\
	&\le& e^{nb^2/(1+2b)} \\
	&\le& e^{nb^2} \,\,=\,\, e^{\tau^2 \log p }
	\eea
		where $b := \tau (\log p/n)^{1/2}$. Thus, \eqref{mixden_zero} holds for some small $\tau>0$.
	It completes the proof for the case of polynomially decreasing $\gamma(k)$.
	
	For the case of exponentially decreasing $\gamma(k)= Ce^{-\beta k}$, consider $k= \min (\log n, p)$ for $\cU_{11}$ instead of $k=\min (n^{1/(2\alpha+1)}, p)$. 
	Then,  similar arguments for the lower bounds of $\cU_{11}$ and $\cU_{12}$ give the desired result.
	
	For the exact banding $\gamma(k)$, consider $\cU_{11}$ with $k=k_0$ and $a=(\log p/n)^{1/2}$, then it completes the proof.
	\quad\quad $\blacksquare$ 
\end{proof}

\begin{proof}[Proof of Theorem \ref{main_pre_linf_LB}.]
	We follow closely the line of a proof in Cai and Zhou (2012a)\nocite{cai2012minimax}. Consider the polynomially decreasing case, $\gamma(k)=Ck^{-\alpha}$, first. 
	Two parameter classes are considered depending on the relation between $p$ and $n$. For $\exp(n^{1/(2\alpha+2)}) \ge p$ case, we show that 
	\bean\label{linf_LB1}
	\inf_{\what{\Omega}} \sup_{\Omega\in \cG_{11}} \bbE_{0n}\|\what{\Omega}-\Omega\|_\infty &\gtrsim& \min \left( n^{-\alpha/(2\alpha+2)}, \frac{p}{\sqrt{n}} \right),
	\eean
	and for $\exp(n^{1/(2\alpha+2)}) \le p$ case, we show that 
	\bean\label{linf_LB2}
	\inf_{\what{\Omega}} \sup_{\Omega\in \cG_{12}} \bbE_{0n}\|\what{\Omega}-\Omega\|_\infty &\gtrsim& \left( \frac{ \log p}{n} \right)^{\alpha/(2\alpha+1)}
	\eean
	for some $\cG_{11} \cup \cG_{12}\subset \cU(\epsilon_0,\gamma)$.
	
	Consider $\exp(n^{1/(2\alpha+2)}) \ge p$ case first.
	Define a class of precision matrices 
	\bea
	\cG_{11} &:=& \bigg\{\Omega(\theta) \in \bbR^{p\times p}: \Omega(\theta)= (I_p-A(\theta))^T (I_p -A(\theta)),\\
	&& \quad\quad A(\theta)=-\tau a \sum_{s=2}^{k}\theta_{s-1} G_s, \theta=(\theta_{s})\in \{0,1\}^{k-1}  \bigg\}
	\eea
	where $G_s := (I(i=s, j=1))$ is a $p\times p$ matrix and $a := n^{-1/2}$ and $k:= \min(n^{1/(2\alpha+2)} ,p)$. It is easy to show that $\cG_{11}\subset \cU(\epsilon_0,\gamma)$ for some small constant $\tau>0$ and all sufficiently large $n$. 
	
	We use the Assouad's lemma,
	\bea
	\inf_{\what{\Omega}} \sup_{\Omega(\theta)\in \cG_{11}} 2\bbE_\theta \|\what{\Omega} - \Omega(\theta)\|_\infty &\ge&
	\min_{H(\theta,\theta')\ge 1} \frac{\|\Omega(\theta)-\Omega(\theta')\|_\infty}{H(\theta,\theta')} \cdot \frac{k-1}{2} \cdot \min_{H(\theta,\theta')=1} \| \bbP_\theta \wedge \bbP_{\theta'}\|.
	\eea
	It is easy to see that 
	\bea
	\min_{H(\theta,\theta')\ge 1} \frac{\|\Omega(\theta)-\Omega(\theta')\|_\infty}{H(\theta,\theta')} &\ge&\tau a.
	\eea
	To show $\min_{H(\theta,\theta')=1} \| \bbP_\theta \wedge \bbP_{\theta'}\| \ge c$ for some $c>0$, it suffices to prove that $\|\bbP_\theta - \bbP_{\theta'}\|_1 \le 1$. Note that
	\bea
	\|\bbP_\theta - \bbP_{\theta'}\|_1^2 &\le& 2K(\bbP_{\theta'} \mid \bbP_\theta) \\
	&\le& C n \|\Omega(\theta')^{-1/2} D_1 \Omega(\theta')^{-1/2}\|_F^2
	\eea
	for some constant $C>0$ where $D_1:= \Omega(\theta)-\Omega(\theta')$. By the same argument used in the proof of Theorem \ref{main_pre_spec_LB}, one can show that $\|\bbP_\theta - \bbP_{\theta'}\|_1^2 \le C' n (\tau a)^2$ for some constant $C'>0$, and it is smaller than 1 for some small constant $\tau>0$. Thus, we have proved the \eqref{linf_LB1} part.
	
	Now consider $\exp(n^{1/(2\alpha+2)}) \le p$ case.
	To show \eqref{linf_LB2} part, define a class of precision matrices 
	\bea
	\cG_{12} &:=& \left\{ \Omega_m \in \bbR^{p\times p}: \Omega_m =(I_p- A_m)^T (I_p-A_m), \, A_m = -\tau B_m \left( \frac{\log p}{nk} \right)^{1/2}  , \, 1\le m\le m_* \right\}
	\eea
	where $B_m:= (I(m+1\le i\le m+k-1,\, j=m) )$ is a $p\times p$ matrix, $m_* = p/k -1$ and $k = (n/\log p)^{1/(2\alpha+1)}$. Without loss of generality, we assume that $p$ can be divided by $k$. By the definition of $\cG_{12}$, tedious calculations yield that $\cG_{12}\subset \cU(\epsilon_0,\gamma)$.
	
	Let $\Omega_0 = I_p$ and $\bbP_m$ be the distribution function of $N(0,\Omega_m^{-1})$ with observation $\bfX_n$.
	It is easy to check that for any $0\le m\neq m' \le m_*$,
	\bea
	\|\Omega_m -\Omega_{m'}\|_\infty &\ge& \tau \left(\frac{k\log p}{n} \right)^{1/2} \,\,=\,\, \tau \left(\frac{\log p}{n} \right)^{\alpha/(2\alpha+1)}
	\eea
	by the definition of $\cG_{12}$ and $k$. Since $k^2 \le p$, for any $1\le m\le m_*$, 
	\bea
	K(\bbP_m\mid \bbP_0)  &\le& Cn\|\Omega_{m'}^{-1/2}D_1 \Omega_{m'}^{-1/2}\|_F^2 \\
	&\le& C'\tau^2 \log p \\
	&\le& c \log m_* 
	\eea
	for some constants $C,C'>0, 0<c<1/8$ and small $\tau>0$, which implies that for any $1\le m\le m_*$, 
	\bea
	\frac{1}{m_*} \sum_{m=1}^{m_*} K(\bbP_m \mid \bbP_0) &\le& c \log m_*
	\eea
	for some $0<c<1/8$, so we can use Fano's lemma,
	\bea
	\inf_{\what{\Omega}} \sup_{\Omega_m\in \cG_{12}} \bbE_m\| \what{\Omega} - \Omega_{m}\|_\infty &\ge& \min_{0\le m\neq m' \le m_*} \frac{\|\Omega_m -\Omega_{m'}\|_\infty}{4} \cdot \frac{\sqrt{m_*}}{1+\sqrt{m_*}} \cdot \left(1-2c- \sqrt{\frac{2c}{\log m_*}} \right).
	\eea
	It completes the proof.
	For more details about Fano's lemma, see Tsybakov (2008)\nocite{tsybakov2008introduction}. 
	
	For the case of exponentially decreasing $\gamma(k)= Ce^{-\beta k}$, consider $k= \min ( [\log n \cdot \log p]^{1/2}, p)$ for $\cG_{11}$ instead of $k=\min (n^{1/(2\alpha+2)}, p)$. Then, similar arguments for the lower bound of $\cG_{11}$ give the desired result. 
	
	For the exact banding $\gamma(k)$, consider $\cG_{11}$ with $k=k_0$ and $a=(\log p/n)^{1/2}$, then it completes the proof.
	\quad\quad $\blacksquare$
\end{proof}

\subsection{Proof of the P-loss Convergence Rates: Theorem \ref{main_pre_spec} and \ref{main_pre_linf}}
Lemma \ref{Nnset}-\ref{Dinv_subG} are used to prove the main theorems.
\begin{lemma}\label{Nnset}
	Let $\bfX_n \overset{iid}{\sim} N_p(0, \Omega_{0,n}^{-1})$ with $\Omega_{0,n} \in \cU(\epsilon_0,\gamma)$,
	\bea
	N_{1n} &:=& \left\{ \bfX_n: \max_j \big\|\what{\V} (Z_{j+1}^{(k+1)}) \big\| \le C_1  \right\},\\
	N_{2n} &:=& \left\{ \bfX_n: \max_j \big\|\what{\V}^{-1}(Z_{j+1}^{(k+1)})\big\| \le C_2  \right\},\\ 
	N_{3n} &:=& \left\{ \bfX_n: \max_j \big\|\what{\V}(Z_{j+1}^{(k+1)})- {\V}(Z_{j+1}^{(k+1)})\big\| \le \sqrt{ C_3(k+\log (n\vee p))/n} \right\},\\
	N_{4n} &:=& \left\{ \bfX_n: \max_j \big\|\what{\V}^{-1}(Z_{j+1}^{(k+1)})- {\V}^{-1}(Z_{j+1}^{(k+1)})\big\| \le \sqrt{ C_4(k+\log (n\vee p))/n} \right\},
	\eea
	where $C_1 = \epsilon_0^{-1}(2+  \sqrt{(k+1)/n})^2, C_2 = 4\epsilon_0^{-1}(1- \sqrt{(k+1)/n})^{-2}, C_4=C_3 C_2^2 \epsilon_0^{-2}$ and $N_n := \bigcap_{j=1}^4 N_{jn}$. If $k+\log p = o(n)$, then for any large constant $C_3$, there exist a positive constant $C_5$ such that
	\bea
	\bbP_{0n} \big( \bfX_n \in N_n^c \big) &\le& 6 p e^{-n(1-\sqrt{(k+1)/n})^2/8} + 4 \cdot 5^k e^{- C_3 C_5 \epsilon_0^{2} (\log (n\vee p) +k)} ,
	\eea
	for all sufficiently large $n$. Here, $C_5$ does not depend on $C_3$.
\end{lemma}

\begin{proof}[Proof of Lemma \ref{Nnset}.]
	We will show that for any large constant $C_3$,
	\bean
	\bbP_{0n}\big( \bfX_n \in N_{1n}^c \big) &\le& 2p e^{-n/2}, \label{n1term} \\
	\bbP_{0n}\big( \bfX_n \in N_{2n}^c \big) &\le& 2p e^{-n(1-\sqrt{(k+1)/n})^2/8}, \label{n2term} \\
	\bbP_{0n}\big( \bfX_n \in N_{3n}^c \big) &\le& 2\cdot 5^k e^{-C_3 C_5 \epsilon_0^{2} (k+\log (n\vee p))}, \label{n3term} \\
	\bbP_{0n}\big( \bfX_n \in N_{4n}^c \big) &\le& 2\cdot 5^k e^{-C_3 C_5 \epsilon_0^{2} (k+\log (n\vee p))} + 2p e^{-n(1-\sqrt{(k+1)/n})^2/8},\label{n4term}
	\eean
	for some positive constants $C_4$ and $C_5$. 
	The inequalities \eqref{n1term} and \eqref{n2term} follow from Lemma A.5 in Lee and Lee (2017). Note that for any large constant $C_3>0$,
	\bean\label{N3prob}
	\bbP_{0n}\big( \bfX_n \in N_{3n}^c \big) 
	&\le& p \cdot 5^{k+1} \left( e^{-C_3 C_6 \epsilon_0^{2} (k+ \log (n\vee p)) } + e^{- C_3^{1/2} C_7 \epsilon_0 \sqrt{n(k+ \log (n\vee p))} } \right) 
	\eean
	for all sufficiently large $n$ and some absolute constants $C_6$ and $C_7$ by Lemma A.4 in Lee and Lee (2017). If we take $C_5=C_6/2$, the right hand side (RHS) of \eqref{N3prob} is bounded by 
	$2 \cdot 5^{k} \exp (-C_3 C_5 \epsilon_0^{2}(k+ \log (n\vee p)))$ for any constant $C_3>0$ and all sufficiently large $n$ because $k+\log(n\vee p) = o(n)$.
	Similarly,
	\bea
	&& \bbP_{0n}\big( \bfX_n \in N_{4n}^c \big) \\
	&\le& \bbP_{0n}\big( \bfX_n \in N_{4n}^c \cap N_{2n} \big) + \bbP_{0n}\big( \bfX_n \in N_{2n}^c \big) \\
	&\le& \bbP_{0n}\left( \max_j \big\|\what{\V}(Z_{j+1}^{(k+1)})- {\V}(Z_{j+1}^{(k+1)})\big\| \ge C_2^{-1}\epsilon_0 \sqrt{ C_4 \frac{k+\log (n\vee p)}{n}} \right) + 2p e^{-n(1-\sqrt{(k+1)/n})^2/8}  \\
	&\le& 2\cdot 5^k e^{-C_3 C_5 \epsilon_0^{2} (k+ \log (n\vee p))}+ 2p e^{-n(1-\sqrt{(k+1)/n})^2/8} 
	\eea
	for $C_4 = C_3 C_2^2 \epsilon_0^{-2}$ and all sufficiently large $n$. Since the inequalities \eqref{n3term} and \eqref{n4term} also hold, this completes the proof. \quad $\blacksquare$
\end{proof}

\begin{lemma}\label{freq_as}
	Consider model \eqref{prec_model} with $\Omega_{0,n} \in \cU(\epsilon_0,\gamma)$ and $\sum_{m=1}^\infty \gamma(m)< \infty$. Define $\what{\Omega}_{nk} := (I_p - \what{A}_{nk})^T \what{D}_{nk}^{-1} (I_p - \what{A}_{nk})$, $\what{D}_{nk}:= diag(\what{d}_{jk})$ and $\what{A}_{nk} := ( \what{a}_{jl}^{(k)} )$, where $\what{a}_{jl}^{(k)} =0$ if $1 \le j \le l \le p$. If  $\Omega_{0,n} \in \cU(\epsilon_0,\gamma)$ and $k^{3/2}(k + \log (n\vee p)) = O(n)$, then 
	\bea
	\bbE_{0n} \left[  \|\what\Omega_{nk}-{\Omega}_{0,n}\|  I(\bfX_n \in N_n)\right] &\lesssim&  k^{3/4} \left[\left(\frac{k+\log  (n\vee p)}{n} \right)^{1/2} + \gamma(k)\right] ,
	\eea
	and if $k(k + \log (n\vee p)) = O(n)$, then
	\bea
	\bbE_{0n} \left[  \|\what\Omega_{nk}-{\Omega}_{0,n}\|_\infty  I(\bfX_n \in N_n)\right] &\lesssim&  k \left[ \left(\frac{k+\log (n\vee p)}{n}\right)^{1/2} + \gamma(k)\right],
	\eea
	where the set $N_n$ is defined at Lemma \ref{Nnset}.
\end{lemma}
\begin{proof}[Proof of Lemma \ref{freq_as}.]
	Define $\Omega_{0,nk} := (I_p - A_{0,nk})^T D_{0,nk}^{-1}(I_p - A_{0,nk})$. 
	Note that
	\bean\label{tria}
	\begin{split}
	&\bbE_{0n} \left[  \|\what\Omega_{nk}-{\Omega}_{0,n}\|  I(\bfX_n \in N_n)\right] \\
	&\le \bbE_{0n} \left[  \|\what\Omega_{nk}-{\Omega}_{0,nk}\|  I(\bfX_n \in N_n)\right] + \| \Omega_{0,nk} - \Omega_{0,n}\| \\
	&\le \bbE_{0n} \left[ \|\what{A}_{nk}^T - A_{0,nk}^T \| \cdot \|D_{0,nk}^{-1}\| \cdot\| I_p- A_{0,nk}\|  I(\bfX_n \in N_n)\right] \\
	&+ \bbE_{0n} \left[ \|\what{D}_{nk}^{-1}- D_{0,nk}^{-1}\|\cdot \|I_p- A_{0,nk}^T\|\cdot \|I_p-A_{0,nk}\|  I(\bfX_n \in N_n)\right]\\
	&+ \bbE_{0n} \left[ \|\what{A}_{nk} - A_{0,nk} \|\cdot \|D_{0,nk}^{-1}\| \cdot\|I_p-A_{0,nk}^T\|  I(\bfX_n \in N_n)\right] \\
	&+ \bbE_{0n} \left[  \|I_p-A_{0,nk}^T\|\cdot \|\what{D}_{nk}^{-1}-D_{0,nk}^{-1}\|\cdot \|\what{A}_{nk}- A_{0,nk}\| I(\bfX_n \in N_n)\right] \\
	&+ \bbE_{0n} \left[ \|D_{0,nk}^{-1}\| \cdot\|\what{A}_{nk}^T - A_{0,nk}^T \| \cdot\|\what{A}_{nk} - A_{0,nk}\|  I(\bfX_n \in N_n)\right] \\
	&+ \bbE_{0n} \left[ \|I_p-A_{0,nk}\|\cdot\|\what{A}_{nk}^T - A_{0,nk}^T \| \cdot\| \what{D}_{nk}^{-1} - D_{0,nk}^{-1}\| I(\bfX_n \in N_n)\right] \\
	&+ \bbE_{0n} \left[ \|\what{A}_{nk}^T - A_{0,nk}^T\|\cdot \|\what{D}_{nk}^{-1}-D_{0,nk}^{-1}\| \cdot\|\what{A}_{nk} - A_{0,nk}\| I(\bfX_n \in N_n)\right] \\
	&+ \| \Omega_{0,nk} - \Omega_{0,n}\|
	\end{split}
	\eean
	by the triangle inequality (See page 223 of Bickel and Levina (2008b)). Also note that
	\bea
	\|I_p - A_{0,nk}\|_\infty &\le& 1 +\|A_{0,nk}- A_{0,n}\|_\infty + \|A_{0,n}\|_\infty \\
	&\le& 1 + C (\sqrt{k}\gamma(k) + 1), \\
	\|I_p - A_{0,nk}\|_1 &\le& 1 +\|A_{0,nk}- A_{0,n}\|_\infty + \|A_{0,n}\|_ 1\\
	&\le& 1 + C k\gamma(k) + \sum_{m=1}^\infty \gamma(m),
	\eea
	for some constant $C>0$ by Lemma \ref{diffAbound}, and $\|D_{0,nk}^{-1}\| \le \max_j \| \V^{-1}(Z_{j+1}^{(k+1)}) \| \le \epsilon_0^{-1}$ using the similar argument to \eqref{D0inv_upperbound}. 	
	If we show that, on $(\bfX_n \in N_n)$,
	\bean
	\| \what{A}_{nk} - A_{0,nk}\|_\infty  &\lesssim& \sqrt{k} \left(\frac{k+\log (n\vee p)}{n}\right)^{1/2}, \label{freq_as_A} \\
	\| \what{A}_{nk} - A_{0,nk}\|_1  &\lesssim& k \left(\frac{k+\log (n\vee p)}{n}\right)^{1/2}, \label{freq_as_A1} \\
	\| \what{D}_{nk}^{-1} - D_{0,nk}^{-1}\|_\infty &\lesssim&  \left(\frac{k+\log (n\vee p)}{n}\right)^{1/2}, \label{freq_as_D}
	\eean
	$\| \Omega_{0,nk} - \Omega_{0,n}\| \lesssim k^{3/4}\gamma(k)$ and $\| \Omega_{0,nk} - \Omega_{0,n}\|_\infty \lesssim k\gamma(k)$, the proof is completed. 
	
	To show \eqref{freq_as_A}, note that
	\bea
	\|\what{A}_{nk} - A_{0,nk}\|_\infty &=&  \max_j \| \what{a}_j^{(k)} - a_{0,j}^{(k)}\|_{1} \\
	&\le& \sqrt{k} \max_j \| \what{a}_j^{(k)} - a_{0,j}^{(k)}\|_{2}\\
	&\le& \sqrt{k}\, \bigg\{ \max_j \left\| \V^{-1}(Z_j^{(k)}) \cdot ( \what{\C}(Z_j^{(k)},X_j) - \C(Z_j^{(k)},X_j) )  \right\|_{2} \\
	&& + \,\, \max_j \left\| ( \what{\V}^{-1}(Z_j^{(k)}) - \V^{-1}(Z_j^{(k)})) \what{\C}(Z_j^{(k)},X_j) \right\|_{2} \bigg\}.
	\eea
	The first part of the last line can be bounded above by
	\bea
	&& \sqrt{k} \max_j \left\| \V^{-1}(Z_j^{(k)}) \cdot ( \what{\C}(Z_j^{(k)},X_j) - \C(Z_j^{(k)},X_j) )  \right\|_{2} \\
	&\le& \sqrt{k} \max_j \left\| \V^{-1}(Z_j^{(k)}) \right\| \left\| \what{\C}(Z_j^{(k)},X_j) - \C(Z_j^{(k)},X_j) \right\|_2 \\
	&\le& \sqrt{k} \max_j \left\| \V^{-1}(Z_j^{(k)}) \right\| \left\| \what{\V}(Z_{j+1}^{(k+1)}) - \V(Z_{j+1}^{(k+1)}) \right\| \\
	&\lesssim& \sqrt{k} \left(\frac{k+ \log (n\vee p)}{n} \right)^{1/2} \quad \text{ on } (\bfX_n \in N_n).
	\eea
	The second part can be bounded similarly
	\bea
	&& \sqrt{k} \max_j \left\| ( \what{\V}^{-1}(Z_j^{(k)}) - \V^{-1}(Z_j^{(k)})) \what{\C}(Z_j^{(k)},X_j) \right\|_{2} \\
	&\le& \sqrt{k} \max_j \left\| \what{\V}^{-1}(Z_j^{(k)}) - \V^{-1}(Z_j^{(k)}) \right\| \left\| \what{\V}(Z_{j+1}^{(k+1)})  \right\|  \\
	&\lesssim& \sqrt{k} \left(\frac{k+ \log (n\vee p)}{n} \right)^{1/2} \quad \text{ on } (\bfX_n \in N_n).
	\eea
	By similar arguments, we can show that the inequality \eqref{freq_as_A1} holds:
	\bea
	\| \what{A}_{nk} - A_{0,nk}\|_1 &\le& k \max_j \| \what{a}_j^{(k)} - a_{0,j}^{(k)}\|_{\max} \\
	&\le& k \max_j \| \what{a}_j^{(k)} - a_{0,j}^{(k)}  \|_2 \\
	&\lesssim& k \left(\frac{k+ \log (n\vee p)}{n} \right)^{1/2} \quad \text{ on } (\bfX_n \in N_n).
	\eea
	To show \eqref{freq_as_D}, note that
	\bea
	\| \what{D}_{nk}^{-1} - D_{0,nk}^{-1}\|_\infty &\le& 
	\|\what{D}_{nk}^{-1}\|_\infty \|D_{0,nk}^{-1}\|_\infty \|\what{D}_{nk}-D_{0,nk}\|_\infty  
	\eea
	where $\|\what{D}_{nk}^{-1}\|_\infty  \|D_{0,nk}^{-1}\|_\infty \le C_2 \epsilon_0^{-1} $ on $(\bfX_n \in N_n)$.
	The rest part is easily bounded above as follows:
	\bea
	\|\what{D}_{nk}-D_{0,nk}\|_\infty &=& \max_j | \what{d}_{jk} - d_{0,jk}| \\
	&\le& \max_j \left|\what{\V}(X_j) - {\V}(X_j) \right| \\
	&+& \max_j \left| \what{\C}(X_j, Z_j^{(k)})\cdot\what{a}_j^{(k)}  - {\C}(X_j, Z_j^{(k)})\cdot {a}_j^{(k)} \right|  \\
	&\le& \max_j \left|\what{\V}(X_j) - {\V}(X_j) \right| + \max_j \left| \what{\C}(X_j,Z_j^{(k)}) \cdot \left
	(\what{a}_j^{(k)} - a_j^{(k)} \right) \right| \\
	&+& \max_j \left| \left(\what{\C}(X_j,Z_j^{(k)}) - \C(X_j,Z_j^{(k)}) \right) a_j^{(k)} \right| \\
	&\lesssim& \left(\frac{k + \log (n\vee p)}{n} \right)^{1/2} \quad \text{ on } (\bfX_n \in N_n) .
	\eea
	
	Hence, by \eqref{tria}, we have shown that
	\bea
	 \bbE_{0n} \left[  \|\what\Omega_{nk}-{\Omega}_{0,nk}\|  I(\bfX_n \in N_n)\right] &\lesssim& k^{3/4}  \left(\frac{k +\log (n\vee p)}{n} \right)^{1/2} +  \| \Omega_{0,nk} - \Omega_{0,n}\|
	\eea
	when $k^{3/2}(k+ \log (n\vee p)) = O(n)$, and
	\bea
	 \bbE_{0n} \left[  \|\what\Omega_{nk}-{\Omega}_{0,nk}\|_\infty  I(\bfX_n \in N_n)\right] &\lesssim& k   \left(\frac{k +\log (n\vee p)}{n} \right)^{1/2} +  \| \Omega_{0,nk} - \Omega_{0,n}\|_\infty
	\eea
	when $k(k+\log (n\vee p)) =O(n)$. The conditions $k^{3/2}(k+ \log (n\vee p)) = O(n)$ and $k(k+\log (n\vee p)) =O(n)$	are required due to the term $\bbE_{0n} \left[ \|D_{0,nk}^{-1}\| \cdot\|\what{A}_{nk}^T - A_{0,nk}^T \| \cdot\|\what{A}_{nk} - A_{0,nk}\|  I(\bfX_n \in N_n)\right]$ in \eqref{tria}.
	
	If we show that $\| \Omega_{0,nk} - \Omega_{0,n}\| \lesssim k^{3/4}\gamma(k)$ and $\| \Omega_{0,nk} - \Omega_{0,n}\|_\infty \lesssim k\gamma(k)$, this completes the proof. 
	By Lemma \ref{diffAbound}, we have $\|A_{0,nk} - A_{0,n}\|_\infty \lesssim \sqrt{k}\gamma(k)$ and $\|A_{0,nk} - A_{0,n}\|_1 \lesssim k\gamma(k)$. Note that
	\bea
	\|D_{0,nk} - D_{0,n}\|_\infty &=& \max_j \left| {a_j^{(k)T}} \V(Z_j^{(k)})a_j^{(k)} - a_j^T \V(Z_j) a_j  \right| \\
	&=& \max_j \left| ((0^T, {a_j^{(k)T}}) - a_j^T) \V(Z_j) \left( \binom{0}{a_j^{(k)}} + a_j \right)  \right|  \\
	&\le&  \|A_{0,nk} - A_{0,n}\|_\infty \max_j \left(\|a_j^{(k)}\|_2+\|a_j\|_2\right) \left\| \V(Z_j) \right\| \\
	&\lesssim& \sqrt{k}\gamma(k).
	\eea
	Thus, it is easy to show that $\| \Omega_{0,nk} - \Omega_{0,n}\| \lesssim k^{3/4}\gamma(k)$ and $\| \Omega_{0,nk} - \Omega_{0,n}\|_\infty \lesssim k\gamma(k)$ by the triangle inequality in \eqref{tria}.  \quad $\blacksquare$
\end{proof}

\begin{lemma}\label{postbound}
	Consider model \eqref{prec_model} and the $k$-BC prior \eqref{kBCprior}. 
	Let 
	\bea
	{\pi}(d_{j}\mid \bfX_n) &:=& IG\left(d_{j}\mid \frac{n_j}{2}, \frac{n}{2}\what{d}_{jk},  d_{j}\le M  \right), \\
	\widetilde{\pi}(d_{j}\mid \bfX_n) &:=& IG\left(d_{j}\mid \frac{n_j}{2}, \frac{n}{2}\what{d}_{jk}  \right),
	\eea
	for $j=1,\ldots,p$. 
	If $M \ge 9\epsilon_0^{-1}$, $\nu_0 = o(n)$ and $k + \log p = o(n)$, then on $(\bfX_n \in N_n)$,
	\bean\label{const_post}
	\begin{split}
	\pi(A_{n}, D_{n} \mid \bfX_n) \,\,&=\,\, \pi(d_{1}\mid \bfX_n) \prod_{j=2}^p \pi(a_j \mid d_{j}, \bfX_n)\pi(d_{j}\mid \bfX_n) \\
	\,\,&\lesssim \,\,  \widetilde{\pi}(d_{1}\mid\bfX_n) \prod_{j=2}^p {\pi}(a_j \mid d_{j},\bfX_n) \widetilde{\pi}(d_{j}\mid \bfX_n) 
	\end{split}
	\eean
	for all sufficiently large $n$, where the set $N_n$ is defined at Lemma \ref{Nnset}.
\end{lemma}
\begin{proof}[Proof of Lemma \ref{postbound}.]
	By the posterior distribution \eqref{post},
	\bea
	\pi(d_{j}\mid \bfX_n) 
	&=& \frac{ IG\left(d_{j}\mid n_j/2, n\what{d}_{jk}/2  \right) I(d_{j}\le M) }{\int_{0}^{M}IG\left({d}'_{j}\mid n_j/2, n\what{d}_{jk}/2 \right)   d {d}'_{j}  }
	\eea
	for $j=1,\ldots,p$. To show \eqref{const_post}, it suffices to prove, on $(\bfX_n \in N_n)$,
	\bea
	\left[ \min_j \widetilde{\pi}( {d}_{j} \le M \mid \bfX_n )  \right]^{-p} &\le& C
	\eea
	for some constant $C>0$. Note that on $(\bfX_n \in N_n)$, $C_1^{-1} \le \what{d}_{jk}^{-1} \le C_2$ and 
	\bea
	\widetilde{\pi}( {d}_{j} \le M \mid \bfX_n ) 
	&=& \widetilde{\pi}( M^{-1}\le {d}_{j}^{-1}  \mid \bfX_n )  \\
	&=& \widetilde{\pi} \left( M^{-1} - \frac{n_j}{n}\what{d}_{jk}^{-1} \le {d}_{j}^{-1} - \frac{n_j}{n}\what{d}_{jk}^{-1}  \mid \bfX_n  \right) \\
	&=& 1- \widetilde{\pi} \left( {d}_{j}^{-1} - \frac{n_j}{n}\what{d}_{jk}^{-1} < M^{-1} - \frac{n_j}{n}\what{d}_{jk}^{-1}   \mid \bfX_n  \right) .
	\eea
	By Boucheron et al. (2013,\nocite{boucheron2013concentration} page 29), if $X$ is a sub-gamma random variable with variance factor $\nu$ and scale parameter $c$, 
	\bean\label{subgam_ineq}
	\max \left[ P(X > \sqrt{2\nu t} + ct ), P(X < -\sqrt{2\nu t} - ct)  \right] &\le& e^{-t}
	\eean
	for all $t>0$. Since a centered $Gamma(a,b)$ random variable is a sub-gamma random variable with $\nu= a/b^2$ and $c = 1/b$, 
	applying $t= n t'$ with $t' = (M -2C_1)^2/(8 M)^2 <1$ to the inequality \eqref{subgam_ineq},
	\bea
	e^{-nt' } &\ge& \widetilde{\pi} \left( {d}_{j}^{-1} - \frac{n_j}{n}\what{d}_{jk}^{-1} < -2 \sqrt{\frac{n_j}{n}} \what{d}_{jk}^{-1}\sqrt{t'} - 2 \what{d}_{jk}^{-1} t' \mid \bfX_n   \right)   \\
	&\ge& \widetilde{\pi} \left( {d}_{j}^{-1} - \frac{n_j}{n}\what{d}_{jk}^{-1} < -4 \what{d}_{jk}^{-1} \sqrt{t'} \mid \bfX_n   \right) \\
	&\ge& \widetilde{\pi} \left( {d}_{j}^{-1} - \frac{n_j}{n}\what{d}_{jk}^{-1} < M^{-1} - \frac{n_j}{n}\what{d}_{jk}^{-1} \mid \bfX_n   \right)
	\eea
	because $M \ge 9 \epsilon_0^{-1} > 2C_1$ for all sufficiently large $n$ and $\nu_0 = o(n)$. 
	Thus, for some constant $C>0$, on $(\bfX_n \in N_n)$,
	\bean\label{djk_concent}
	\widetilde{\pi}( {d}_{j} \le M \mid \bfX_n ) 
	&\ge& 1 - e^{-Cn} ,
	\eean
	and
	\bea
	\left[ \min_j \widetilde{\pi}( {d}_{j} \le M \mid \bfX_n )  \right]^{-p} &\le& (1 - e^{-Cn})^{-p} \\
	&=& (1 - e^{-Cn})^{-e^{Cn}\cdot p/e^{Cn} }\\
	&\le& (C')^{p/e^{Cn}} \,\,\lra\,\, 1
	\eea
	as $n\to\infty$ for some constant $C'>0$. \quad $\blacksquare$
\end{proof}

\begin{lemma}\label{Ank_post}
	Consider the model \eqref{prec_model} and the $k$-BC prior \eqref{kBCprior} with $M\ge 9\epsilon_0^{-1}$ and $\nu_0 = o(n)$. If $k+\log p = o(n)$, then
	\bea
	\bbE^\pi \left( \|A_{n} - \what{A}_{nk}\|_\infty^2 \mid \bfX_n  \right) &\le& C k \left(\frac{k + \log p}{n}  \right)  \quad \text{ on } (\bfX_n \in N_n), \\
	\bbE^\pi \left( \|A_{n} - \what{A}_{nk}\|_1^2 \mid \bfX_n  \right) &\le& C k \left(\frac{k + \log p}{n}  \right)  \quad \text{ on } (\bfX_n \in N_n),
	\eea
	for some constant $C>0$ and all sufficiently large $n$.
\end{lemma}

\begin{proof}[Proof of Lemma \ref{Ank_post}.]
	Let $\bbE^{\widetilde{\pi}}(\cdot \mid \bfX_n)$ denote the expectation with respect to $\widetilde{\pi}(d_{1}\mid \bfX_n) \prod_{j=2}^p {\pi}(a_j \mid d_{j}, \bfX_n) \widetilde{\pi}(d_{j}\mid \bfX_n)$ in Lemma \ref{postbound}.
	Note that on $(\bfX_n \in N_n)$,
	\bea
	&&\bbE^\pi \left( \|A_{n} - \what{A}_{nk}\|_\infty^2 \mid \bfX_n  \right) \\
	&\le& k  \cdot \bbE^\pi \left( \max_j \| a_j - \what{a}_j^{(k)}\|_2^2  \mid \bfX_n  \right) \\
	&\le& k \cdot \bbE^{\pi} \left( \max_j \frac{d_{j}}{n} \Big\| \what{\V}^{-1}(Z_j^{(k)})\Big\| \cdot \Big\| \sqrt{\frac{n}{d_{j}}} \cdot \what{\V}^{1/2}(Z_j^{(k)})\cdot (a_j - \what{a}_j^{(k)})\Big\|_2^2  \mid \bfX_n  \right)\\
	&\le& \frac{k M C_2}{n} \cdot \bbE^{\pi} \left( \max_j   \Big\| \sqrt{\frac{n}{d_{j}}} \cdot \what{\V}^{1/2}(Z_j^{(k)})\cdot (a_j - \what{a}_j^{(k)})\Big\|_2^2  \mid \bfX_n  \right) \\
	&\lesssim& \frac{k}{n} \cdot \bbE^{\widetilde{\pi}} \left( \max_j  \Big\| \sqrt{\frac{n}{d_{j}}} \cdot \what{\V}^{1/2}(Z_j^{(k)})\cdot (a_j - \what{a}_j^{(k)})\Big\|_2^2  \mid \bfX_n  \right) \\
	&=& \frac{k}{n} \cdot \bbE \left( \max_j \chi_{jk}^2 \right)
	\eea
	by Lemma \ref{postbound}.  $\chi_{jk}^2$ is a chi-square random variable with $k_j := \min(j-1,k)$ degree of freedom. By the maximal inequality for chi-square random variables (Boucheron et al., 2013, Example 2.7), 
	\bea
	\bbE \left( \max_j \chi_{jk}^2 \right) &=& k_j + \bbE \left( \max_j \chi_{jk}^2 - k_j \right) \\
	&\le& C \left( k  + \log p \right)
	\eea
	for some constant $C>0$. Thus, we have
	\bea
	\bbE^\pi \left( \|A_{n} - \what{A}_{nk}\|_\infty^2 \mid \bfX_n  \right)
	&\le& Ck \left(\frac{k+ \log p}{n}  \right)
	\eea
	on $(\bfX_n \in N_n)$, for some constant $C>0$. 
	
	Let $a_{c_j}:= (a_{j+1,j}, \ldots, a_{\min(j+k,p),j} )^T$ be the nonzero column vector of $A_{n}$. Since the posterior distributions for $a_{c_j}$'s are the independent multivariate normal distributions with finite variances whose rate is $1/n$ on $(\bfX_n \in N_n)$, it is easy to show that
	\bea
	\bbE^\pi \left( \|A_{n} - \what{A}_{nk}\|_1^2 \mid \bfX_n  \right)
	&\le& Ck \left(\frac{k+ \log p}{n}  \right)
	\eea 
	on $(\bfX_n \in N_n)$, for some constant $C>0$ using similar arguments. \quad $\blacksquare$
\end{proof}

\begin{lemma}\label{Dinv_subG}
	Consider the model \eqref{prec_model} and the $k$-BC prior \eqref{kBCprior} with $M\ge 9\epsilon_0^{-1}$ and $\nu_0 = o(n)$. If $k+\log p = o(n)$ and $k^2 = O(n \log p)$, then
	\bea
	\bbE^\pi \left(\|D_{n}^{-1}- \what{D}_{nk}^{-1}\|_\infty \mid \bfX_n \right) &\le& C \left(\frac{\log p}{n}\right)^{1/2} \quad \text{ on } (\bfX_n \in N_n)
	\eea
	for some constant $C>0$ and all sufficiently large $n$.
\end{lemma}
\begin{proof}[Proof of Lemma \ref{Dinv_subG}]
	By Lemma \ref{postbound}, on $(\bfX_n \in N_n)$,
	\bea
	\bbE^\pi \left(\|D_{n}^{-1}- \what{D}_{nk}^{-1}\|_\infty \mid \bfX_n \right) 
	&\le& C \bbE^{\widetilde{\pi}} \left(\|D_{n}^{-1}- \what{D}_{nk}^{-1}\|_\infty \mid \bfX_n \right)
	\eea
	for some constant $C>0$. It is easy to show that 
	\bea
	\bbE^{\widetilde{\pi}} \left(\|D_{n}^{-1}- \what{D}_{nk}^{-1}\|_\infty \mid \bfX_n \right)
	&\le& \bbE^{\widetilde{\pi}} \left( \max_j \left| d_{j}^{-1} - \frac{n_j}{n}\what{d}_{jk}^{-1} \right| \mid \bfX_n \right)
	+ \max_j \left| \frac{n-n_j}{n} \what{d}_{jk}^{-1}  \right| \\
	&\le& \frac{1}{\lambda} \log \exp \bbE^{\widetilde{\pi}} \left( \lambda \max_j \left| d_{j}^{-1} - \frac{n_j}{n}\what{d}_{jk}^{-1} \right| \mid \bfX_n \right) + \frac{2k}{n}C_2 \\
	&\le& \frac{1}{\lambda} \log  \bbE^{\widetilde{\pi}} \left(  \max_j e^{ \lambda | d_{j}^{-1} - \frac{n_j}{n}\what{d}_{jk}^{-1}  |} \mid \bfX_n \right) + \frac{2k}{n}C_2  \\
	&\le& \frac{1}{\lambda} \log \left[p\cdot \max_j \bbE^{\widetilde{\pi}} \left(   e^{ \lambda | d_{j}^{-1} - \frac{n_j}{n}\what{d}_{jk}^{-1}  |} \mid \bfX_n \right)\right] + \frac{2k}{n}C_2 
	\eea
	for any $\lambda>0$, on $(\bfX_n \in N_n)$. Let $\lambda < n\what{d}_{jk}/2$.
	Note that the upper bound for the moment generating function of $| d_{j}^{-1} - n_j\what{d}_{jk}^{-1}/n |$ is given by
	\bea
	\bbE^{\widetilde{\pi}} \left(   e^{ \lambda | d_{j}^{-1} - \frac{n_j}{n}\what{d}_{jk}^{-1}  |} \mid \bfX_n \right) 
	&=& \int_0^\infty e^{ \lambda | d_{j}^{-1} - \frac{n_j}{n}\what{d}_{jk}^{-1}  |} Gamma\left( d_{j}^{-1}\mid \frac{n_j}{2}, \frac{n}{2}\what{d}_{jk} \right) d d_{j}^{-1} \\
	&\le& \int_0^{n_j \what{d}_{jk}^{-1}/n } e^{ \lambda (\frac{n_j}{n}\what{d}_{jk}^{-1} -  d_{j}^{-1}  )}Gamma\left( d_{j}^{-1}\mid \frac{n_j}{2}, \frac{n}{2}\what{d}_{jk} \right) d d_{j}^{-1} \\
	&+& \bbE^{\widetilde{\pi}} \left(   e^{ \lambda ( d_{j}^{-1} - \frac{n_j}{n}\what{d}_{jk}^{-1} )} \mid \bfX_n \right) \\
	&\le& e^{ \lambda \frac{n_j}{n}\what{d}_{jk}^{-1} } \int_0^{\infty} e^{ -\lambda  d_{j}^{-1} }Gamma\left( d_{j}^{-1}\mid \frac{n_j}{2}, \frac{n}{2}\what{d}_{jk} \right) d d_{j}^{-1} \\
	&+& \exp \left( \frac{n_j \lambda^2}{n \what{d}_{jk} (n \what{d}_{jk}-2\lambda)  } \right) \\
	&\le& e^{ \lambda \frac{n_j}{n}\what{d}_{jk}^{-1} } \left( \frac{n\what{d}_{jk}}{n \what{d}_{jk}+2\lambda} \right)^{n_j/2} + \exp \left( \frac{n_j \lambda^2}{n \what{d}_{jk} (n \what{d}_{jk}-2\lambda)  } \right).
	\eea
	The second inequality follow from page 28 of Boucheron et al. (2013)\nocite{boucheron2013concentration}. Since  $\lambda < n\what{d}_{jk}/2$, 
	\bea
	e^{ \lambda \frac{n_j}{n}\what{d}_{jk}^{-1} } \left( \frac{n\what{d}_{jk}}{n \what{d}_{jk}+2\lambda} \right)^{n_j/2}
	&=& e^{ \lambda n_j/(n\what{d}_{jk}) } \left( 1 + \frac{2\lambda}{n \what{d}_{jk}} \right)^{-n_j/2} \\
	&\le& \left( 1 + \frac{2\lambda }{n \what{d}_{jk}} \right)^{\lambda n_j /(2n \what{d}_{jk})} \\
	&=& \left( 1 + \frac{2\lambda }{n \what{d}_{jk}} \right)^{n\what{d}_{jk}/(2\lambda) \cdot \lambda^2 n_j /(n^2 \what{d}_{jk}^2) } \\
	&\le& \exp \left( \frac{\lambda^2 n_j}{n^2 \what{d}_{jk}^2} \right),
	\eea
	where the first inequality follows from Lemma \ref{expineq}. Thus, on $(\bfX_n \in N_n)$,
	\bea
	\bbE^{\widetilde{\pi}} \left(\|D_{n}^{-1}- \what{D}_{nk}^{-1}\|_\infty \mid \bfX_n \right)
	&\le& \frac{1}{\lambda} \log \left[p\cdot \max_j \bbE^{\widetilde{\pi}} \left(   e^{ \lambda | d_{j}^{-1} - \frac{n_j}{n}\what{d}_{jk}^{-1}  |} \mid \bfX_n \right)\right] + \frac{2k}{n}C_2 \\
	&\le& \frac{\log p}{\lambda} + \frac{1}{\lambda}\max_j \log \left[\exp \left( \frac{\lambda^2 n_j}{n^2 \what{d}_{jk}^2} \right) +  \exp \left( \frac{n_j \lambda^2}{n \what{d}_{jk} (n \what{d}_{jk}-2\lambda)  } \right)\right] \\
	&+& \frac{2k}{n}C_2 \\
	&\le& \frac{\log p}{\lambda} + \frac{2\log 2}{\lambda} + \max_j \left( \frac{\lambda n_j}{n^2 \what{d}_{jk}^2} + \frac{n_j \lambda}{n \what{d}_{jk} (n \what{d}_{jk}-2\lambda)  } \right) + \frac{2k}{n}C_2 \\
	&\le& \frac{\log p}{\lambda} + \frac{2\log 2}{\lambda} +  \frac{\lambda C_2^{2}}{n } + \frac{\lambda C_2}{  (n C_2^{-1} -2\lambda)  } + \frac{2k}{n}C_2  \\
	&\le& C \left( \frac{\log p}{n} \right)^{1/2}
	\eea
	for some constant $C>0$ if we choose $\lambda \asymp (n \log p)^{1/2}$. \quad $\blacksquare$
\end{proof}

\begin{proof}[Proof of Theorem \ref{main_pre_spec}]
	Note that 
	\bean
	\bbE_{0n}\bbE^\pi \left( \|\Omega_{n}-\Omega_{0,n}\| \mid \bfX_n \right) &\le&  \bbE_{0n} \left[ \bbE^\pi \left( \|\Omega_{n}-{\Omega}_{0,n}\| \mid \bfX_n \right)I(\bfX_n \in N_n)\right] \label{epr1} \\
	&+& \bbE_{0n} \left[ \bbE^\pi \left( \|\Omega_{n}-{\Omega}_{0,n}\| \mid \bfX_n \right)I(\bfX_n \in N_n^c)\right] \label{epr2}
	\eean
	where the set $N_n$ is defined at Lemma \ref{Nnset}. 
	The term \eqref{epr2} is bounded above by
	\bea
	&&\bbE_{0n} \left[ \left( \bbE^\pi ( \|\Omega_{n}\| \mid \bfX_n) + \|\Omega_{0,n}\| \right)  I(\bfX_n \in N_n^c) \right]\\
	&\le& \bbE_{0n} \left[ \left( \bbE^\pi ( \|I_p-A_{n}\|_1 \|I_p- A_{n}\|_\infty \|D_{n}^{-1}\|  \mid \bfX_n) + \|\Omega_{0,n}\| \right)  I(\bfX_n \in N_n^c) \right] \\
	&\le& \left\{ \bbE_{0n} \left[  \bbE^\pi ( \|I_p-A_{n}\|_1 \|I_p- A_{n}\|_\infty \|D_{n}^{-1}\|  \mid \bfX_n)  \right]^2 \right\}^{1/2} \bbP_{0n} ( \bfX_n \in N_n^c)^{1/2}  \\
	&+&  \|\Omega_{0,n}\|_\infty \bbP_{0n} ( \bfX_n \in N_n^c) \\
	&\le& \left( p^\kappa +  C \right)\cdot \left(6 p e^{-n(1-\sqrt{(k+1)/n})^2/8} + 4 \cdot 5^k e^{- C_3 C_5 \epsilon_0^{2} (k+ \log (n\vee p))} \right)^{1/2} \\
	&\lesssim& n^{-1}
	\eea
	for all sufficiently large $n$ and some positive constants $\kappa, C_3$ and $C_5$.
	The third inequality follows from Lemma \ref{Al1norm} and Lemma \ref{Nnset}. 
	
	We decompose the term \eqref{epr1} as follows:
	\bean
	&&\bbE_{0n} \left[ \bbE^\pi \left( \|\Omega_{n}-{\Omega}_{0,n}\| \mid \bfX_n \right)I(\bfX_n \in N_n)\right]\nonumber\\
	&\le& \bbE_{0n} \left[ \bbE^\pi \left( \|\Omega_{n}-\what{\Omega}_{nk}\| \mid \bfX_n \right)I(\bfX_n \in N_n)\right] \label{bayes1} \\
	&+& \bbE_{0n} \left[  \|\what\Omega_{nk}-{\Omega}_{0,n}\|  I(\bfX_n \in N_n)\right]. \label{freq1}
	\eean
	By Lemma \ref{freq_as}, the upper bound for \eqref{freq1} is $Ck^{3/4}[((k + \log (n\vee p))/n)^{1/2} + \gamma(k)]$ for some constant $C>0$ because we assume that $k^{3/2}(k +\log (n\vee p))= O(n)$. Note that the term \eqref{bayes1} can be decomposed as \eqref{tria} and
	\bea
	\|I_p - \what{A}_{nk}\|_1 &\le& \| I_p - A_{0,nk}\|_1 + \| \what{A}_{nk} - A_{0,nk}\|_1 \\
	&\le& 1 + \sum_{m=1}^\infty \gamma(m) + Ck \gamma(k) + C k\left(\frac{k + \log (n\vee p)}{n} \right)^{1/2},\\
	\|I_p - \what{A}_{nk}\|_\infty &\le& \| I_p - A_{0,nk}\|_\infty + \| \what{A}_{nk} - A_{0,nk}\|_\infty \\
	&\le& 1 + \gamma(1) + C\sqrt{k} \gamma(k)  + C k^{1/2}\left(\frac{k + \log (n\vee p)}{n} \right)^{1/2},  \\
	\| I_p - \what{A}_{nk}\| &\le& \|I_p - A_{0,nk}\| + \|\what{A}_{nk} - A_{0,nk} \| \\
	&\le& 1 + \sum_{m=1}^\infty \gamma(m) + C k^{3/4}\gamma(k)  + C k^{3/4}  \left(\frac{k + \log (n\vee p)}{n} \right)^{1/2}
	\eea
	and $\|\what{D}_{nk}^{-1}\| \le C_2$ on $(\bfX_n \in N_n)$ for some constant $C>0$. By Lemma \ref{Ank_post} and Lemma \ref{Dinv_subG}, it is easy to show that the upper bound for \eqref{bayes1} is $Ck^{1/2}((k+\log (n\vee p))/n)^{1/2}$ for some constant $C>0$ because we assume that $k^{3/2}(k +\log (n\vee p))= O(n)$. \quad $\blacksquare$
\end{proof}

\begin{proof}[Proof of Theorem \ref{main_pre_linf}]
	We can use the same arguments used in the proof of Theorem \ref{main_pre_spec}. 
	It suffices to prove that 
	\bea
	\|I_p - \what{A}_{nk}\|_1 &\lesssim& \sqrt{k} \text{ on } (\bfX_n \in N_n).
	\eea
	It trivially holds because we assume that $k(k +\log (n\vee p))= O(n)$. \quad $\blacksquare$
\end{proof}

\subsection{Proof of Corollary \ref{pluginmean}}
Lemma \ref{dhat_ahat} is used to prove Theorem \ref{pluginmean}.
\begin{lemma}\label{dhat_ahat}
	Consider the model \eqref{prec_model} and $\Omega_{0,n}\in \cU(\epsilon_0,\gamma)$. Let $\what{d}_{jk}$ and $\what{a}_{ji}^{(k)}$ be defined as before. If $k=o(n)$, then for given positive integer $m$,
	\bea
	\bbE_{0n}(\what{d}_{jk}^{-m} ) &\lesssim&  (k+1)^{m+1}  , \\
	\bbE_{0n}((\what{a}_{ji}^{(k)})^m ) &\lesssim& (k+1)^{2m+1}.
	\eea
\end{lemma}
\begin{proof}
	Note that
	\bea
	\bbE_{0n}(\what{d}_{jk}^{-m} ) 
	&\le& \bbE_{0n} \| \what{\V}^{-1}(Z_{j+1}^{(k+1)})\|^m \\
	&\le& \bbE_{0n} \left[tr\left( \what{\V}^{-1}(Z_{j+1}^{(k+1)}) \right)\right]^m \\
	&\le& (k+1)^m \sum_{l=1}^{k+1} \bbE_{0n}  \left[ \what{\V}^{-1}(Z_{j+1}^{(k+1)})_{(l)}\right]^m
	\eea
	where for any $p\times p$ matrix $A$, $A_{(i)}$ is the $(i,i)$ component of $A$. Also note that $[ \what{\V}^{-1}(Z_{j+1}^{(k+1)})]_{(l)}$ is a inverse-gamma distribution $IG((n-k)/2, n[ \V^{-1}(Z_{j+1}^{(k+1)})]_{(l)}/2)$ because diagonal elements of a inverse-Wishart matrix are inverse-gamma random variables (Huang and Wand, 2013)\nocite{huang2013simple}.
	Since $\Omega_{0,n}\in \cU(\epsilon_0,\gamma)$,
	\bea
	(k+1)^m \sum_l \bbE_{0n}  \left[ \what{\V}^{-1}(Z_{j+1}^{(k+1)})_{(l)}\right]^m &\le& (k+1)^{m+1} \left(\frac{n\epsilon_0^{-1}}{n-k-2m} \right)^m \\
	&\lesssim& (k+1)^{m+1}.
	\eea
	Similarly,
	\bea
	\bbE_{0n}((\what{a}_{ji}^{(k)})^m ) &\le& \bbE_{0n}\left[ \|\what{\V}^{-1}(Z_{j+1}^{(k+1)})\|^m \|\what{\V}(Z_{j+1}^{(k+1)})\|^m \right]\\
	&\le& \bbE_{0n} \left\{ \left[tr\left( \what{\V}^{-1}(Z_{j+1}^{(k+1)})\right)\right]^m  \left[tr\left( \what{\V}(Z_{j+1}^{(k+1)})\right)\right]^m \right\} \\
	&\le& \left\{ \bbE_{0n} \left[tr \left( \what{\V}^{-1}(Z_{j+1}^{(k+1)})\right)\right]^{2m}  \bbE_{0n} \left[tr\left( \what{\V}(Z_{j+1}^{(k+1)})\right)\right]^{2m}  \right\}^{1/2}\\
	&\lesssim& (k+1)^{2m+1} 
	\eea
	because diagonal elements of a Wishart matrix are gamma random variables (Rao, 2009)\nocite{rao2009linear}, i.e. $[\what{\V}(Z_{j+1}^{(k+1)})]_{(l)} \sim  Gamma(n/2, n [\V(Z_{j+1}^{(k+1)})]_{(l)}^{-1}/2)$. \quad $\blacksquare$
\end{proof}

\begin{proof}[Proof of Corollary \ref{pluginmean}.]
	Since 
	\bea
	\bbE_{0n} \|\what{\Omega}_{nk}^{LL} - \Omega_{0,n}\| &\le& \bbE_{0n}\| \bbE^\pi(\Omega_{n}\mid \bfX_n) - \Omega_{0,n} \|  + \bbE_{0n}\| \bbE^\pi(\Omega_{n}\mid \bfX_n) - \what{\Omega}_{nk}^{LL} \| \\
	&\le& \bbE_{0n} \bbE^\pi(\|\Omega_{n} - \Omega_{0,n}\| \mid \bfX_n ) + \bbE_{0n}\| \bbE^\pi(\Omega_{n}\mid \bfX_n) - \what{\Omega}_{nk}^{LL} \|,
	\eea
	it suffices to prove
	\bea
	\bbE_{0n}\| \bbE^\pi(\Omega_{n}\mid \bfX_n) - \what{\Omega}_{nk}^{LL} \|_\infty  &\le& \frac{Ck^2}{n} \\
	&\le& k^{3/4}\left[ \left(\frac{k + \log(n\vee p)}{n} \right)^{1/2} + \gamma(k) \right]
	\eea
	for some constant $C>0$ because of the assumption $k(k + \log(n \vee p)) = O(n)$.
	
	
	Let $\what{\Omega}_{nk}^{LL} = ( \what{\Omega}_{ij}^{LL})$, then for $i<j\le i+k$,
	\bean
	\bbE_{0n}\left| \bbE^\pi(\Omega_{n,ij}\mid \bfX_n) - \what{\Omega}_{ij}^{LL}  \right| 
	&\le& 
	\bbE_{0n} \left| \bbE^\pi(d_{j}^{-1}a_{ji} \mid \bfX_n)  - \frac{n_j}{n}\what{d}_{jk}^{-1} \what{a}_{ji}^{(k)}  \right| \label{diffmean1} \\
	&+& \sum_{l=j+1}^{i+k} \bbE_{0n} \left|\bbE^\pi(d_{l}^{-1} a_{li} a_{lj} \mid \bfX_n)  - \frac{n_l}{n}\what{d}_{lk}^{-1}\what{a}_{li}^{(k)} \what{a}_{lj}^{(k)}  \right| \label{diffmean2}
	\eean
	by \eqref{omega_ij}.
	The \eqref{diffmean1} term can be decomposed by
	\bean
	&&\bbE_{0n} \left| \left( \bbE^\pi(d_{j}^{-1}a_{ji} \mid \bfX_n)  - \frac{n_j}{n}\what{d}_{jk}^{-1} \what{a}_{ji}^{(k)} \right) I(\bfX_n \in N_n)  \right| \label{diffmean11} \\
	&+& \bbE_{0n} \left| \left( \bbE^\pi(d_{j}^{-1}a_{ji} \mid \bfX_n)  - \frac{n_j}{n}\what{d}_{jk}^{-1} \what{a}_{ji}^{(k)} \right) I(\bfX_n \in N_n^c)  \right| \label{diffmean12}.
	\eean
	To deal with the above terms, we need to compute the expectation of truncated distributions. 
	When $Y$ is a truncated gamma distribution $Y \sim Gamma^{Tr}(\alpha, \beta, c_1 \le Y \le c_2)$, the expectation of $Y$ is 
	\bea
	\bbE Y &=& \frac{\alpha}{\beta} \frac{\int_{c_1}^{c_2}Gamma(y \mid \alpha+1, \beta) dy }{\int_{c_1}^{c_2}Gamma(y \mid \alpha, \beta) dy}
	\eea 
	(Coffey and Muller, 2000)\nocite{coffey2000properties}.
	Thus, one can show that \eqref{diffmean11} is bounded above by
	\bea
	&&\bbE_{0n}\left| \frac{n_j}{n}\what{d}_{jk}^{-1}\what{a}_{ji}^{(k)} \left( \frac{\int_{0}^{M} Gamma(d_{j}^{-1}\mid \frac{n_j}{2}+1, \frac{n}{2}\what{d}_{jk}) d d_{j}^{-1}  }{\int_{0}^{M} Gamma(d_{j}^{-1}\mid \frac{n_j}{2}, \frac{n}{2}\what{d}_{jk}) d d_{j}^{-1} } - 1 \right) I(\bfX_n\in N_n)  \right|  \\
	&\le& C_1 C_2^2 e^{-cn} 
	\eea
	for all sufficiently large $n$ and some positive constant $c$ by the same argument with \eqref{djk_concent}. 
	On the other hand, \eqref{diffmean12} is bounded above by
	\bea
	&& C \left[\bbE_{0n}(\what{d}_{jk}^{-2} (\what{a}_{ji}^{(k)})^2 ) \right]^{1/2}  \bbP_{0n}(\bfX_n \in N_n^c)\\
	&\lesssim& (k+1)^{7/2} \bbP_{0n}(\bfX_n \in N_n^c) \\
	&\le& \frac{1}{n^2}
	\eea
	for some constant $C>0$ and all sufficiently large $n$ by Lemma \ref{Nnset}, Lemma \ref{dhat_ahat} and the choice of large $C_3$ in the set $N_n$. 
	
	The \eqref{diffmean2} can be decomposed by
	\bean
	&& \sum_{l=j+1}^{i+k} \bbE_{0n} \left|\left(\bbE^\pi(d_{l}^{-1} a_{li} a_{lj} \mid \bfX_n)  - \frac{n_l}{n}\what{d}_{lk}^{-1}\what{a}_{li}^{(k)} \what{a}_{lj}^{(k)} \right) I(\bfX_n \in N_n) \right| \label{diffmean21} \\
	&+& \sum_{l=j+1}^{i+k} \bbE_{0n} \left|\left(\bbE^\pi(d_{l}^{-1} a_{li} a_{lj} \mid \bfX_n)  - \frac{n_l}{n}\what{d}_{lk}^{-1}\what{a}_{li}^{(k)} \what{a}_{lj}^{(k)} \right) I(\bfX_n \in N_n^c) \right|. \label{diffmean22}
	\eean
	Note that in \eqref{diffmean21},
	\bea
	\bbE^\pi(d_{l}^{-1} a_{li} a_{lj} \mid \bfX_n) &=& \bbE^\pi(d_{l}^{-1} \bbE^\pi (a_{li} a_{lj}\mid d_{l},\bfX_n) \mid \bfX_n) \\
	&=& \bbE^\pi(d_{l}^{-1} \bbE^\pi (a_{li}\mid d_{l},\bfX_n) \bbE^\pi( a_{lj} \mid d_{l},\bfX_n) \mid \bfX_n) \\
	&+& \bbE^\pi(d_{l}^{-1} \C^\pi (a_{li}, a_{lj}\mid d_{l},\bfX_n) \mid \bfX_n).
	\eea
	If we prove that $\sum_{l=j+1}^{i+k} \bbE_{0n}|\bbE^\pi(d_{l}^{-1} \C^\pi (a_{li}, a_{lj} \mid d_{l},\bfX_n) \mid \bfX_n)I(\bfX_n \in N_n)|  \lesssim k/n$, \eqref{diffmean21} is bounded above by $Ck/n$ for some constant $C>0$ by the similar arguments used in \eqref{diffmean11}. It is easy to show that
	\bea
	&& \sum_{l=j+1}^{i+k} \bbE_{0n} \left|\bbE^\pi(d_{l}^{-1} \C^\pi (a_{li}, a_{lj} | d_{l},\bfX_n) \mid \bfX_n)I(\bfX_n \in N_n)  \right|\\
	&\le& \sum_{l=j+1}^{i+k} \bbE_{0n} \left[ \bbE^\pi \left(d_{l}^{-1} \big|\C^\pi (a_{li}, a_{lj} | d_{l},\bfX_n)\big| \,\,\bigg|\,\, \bfX_n \right)   I(\bfX_n\in N_n) \right] \\
	&\le& \sum_{l=j+1}^{i+k} \bbE_{0n} \left(  \bbE^\pi \left( d_{l}^{-1} \left[ \V^\pi(a_{li} | d_{l},\bfX_n)  \V^\pi(a_{lj} | d_{l},\bfX_n)  \right]^{1/2}  \Big|\, \bfX_n \right)   I(\bfX_n\in N_n) \right) \\
	&\lesssim& \frac{k}{n}.
	\eea
	Similar to \eqref{diffmean12}, \eqref{diffmean22} is bounded above by $C/n^2$ for some constant $C>0$. 
	Thus, we have shown 
	\bea
	\bbE_{0n} \left|\bbE^\pi(\Omega_{n,ij}\mid \bfX_n) - \what{\Omega}_{ij}^{LL}  \right| &\lesssim& \frac{k}{n}
	\eea
	for any $i < j \le i+k$.
	Since $\Omega_{n,ii} = d_{i}^{-1} + \sum_{l=i+1}^{i+k}d_{l}^{-1} a_{li}^2$ for $i<p$ and $\Omega_{n,pp} = d_{p}^{-1}$, 
	\bea
	\bbE_{0n} \left|\bbE^\pi(\Omega_{n,ii}\mid \bfX_n) - \what{\Omega}_{ii}^{LL}  \right| &\lesssim& \frac{k}{n}
	\eea
	can be shown easily for $1\le i \le p$ by similar arguments.
	Thus, it implies 
	\bea
	\bbE_{0n}\| \bbE^\pi(\Omega_{n}\mid \bfX_n) - \what{\Omega}_{nk}^{LL} \|_\infty  &\lesssim& \frac{k^2}{n}.\quad \blacksquare
	\eea
\end{proof}

\appendix
\section{Appendix: auxiliary results}
\begin{lemma}\label{expineq}
	For any $x, n>0$,  
	\bea
	e^x &\le& \left(1+ \frac{x}{n}\right)^{n + x/2}.
	\eea
\end{lemma}
The proof can be obtained by a simple algebra. 

\begin{lemma}\label{Al1norm}
	If we assume that $\Omega_{0,n}\in \cU(\epsilon_0,\gamma)$ and $\sum_{k=1}^\infty\gamma(k) < \infty$, then
	\bea
	\|\Omega_{0,n}\|_\infty &<& C
	\eea
	for some $C>0$  not depending on $p$. 
\end{lemma}

\begin{proof}[Proof]
	Let $\Omega_{0,n} =(I_p - A_{0,n})^T D_{0,n}^{-1} (I_p -A_{0,n})$ be the MCD of $\Omega_{0,n}$. Since $\|\Omega_{0,n}\|_\infty \le \|I_p-A_{0,n}\|_1 \|D_{0,n}^{-1}\|_\infty \|I_p-A_{0,n}\|_\infty$ 
	and
	\begin{align}
	\|I_p-A_{0,n}\|_\infty &\le 1 + \| A_{0,n}\|_\infty \le 1 + \gamma(1), \nonumber \\
	\begin{split}
	\|D_{0,n}^{-1}\|_\infty &= \max_j d_j^{-1} \\
	&= \max_j  \Big\|\V^{1/2}(Z_{j+1})\cdot \binom{-a_j}{1}  \Big\|_2^{-2} \\
	&\le \max_j \lambda_{\min}\left(\V(Z_{j+1}) \right)^{-1} = \max_j \left\| \V^{-1}(Z_{j+1}) \right\| \le \epsilon_0^{-1},
	\end{split}\label{D0inv_upperbound}
	\end{align}
	 we only need to prove $\|A_{0,n}\|_1 \le C$ for some $C>0$. By the definition of $\cU(\epsilon_0,\gamma)$, it is easy to show $|a_{ij}| \le \gamma(i-j)$ for all $i>j$. Thus,
	\bea
	\|A_{0,n}\|_1 &=& \max_j \sum_{i=j+1}^p |a_{ij}|\\
	&\le& \max_j \sum_{i=j+1}^{p} \gamma(i-j) \\
	&\le& \sum_{m=1}^\infty \gamma(m)\,\, < \,\, \infty. \quad \blacksquare
	\eea
\end{proof}

\begin{lemma}\label{spec_partitioned}
	For any positive integers $p_1$ and $p_2$, let $A_{11}, A_{12}$ and $A_{22}$ be a $p_1\times p_1, p_1\times p_2$ and $p_2\times p_2$ matrix, 
	$$ \|A_{12}\| \leq \left\|  \begin{pmatrix}
	A_{11} & A_{12} \\
	A_{12}^T & A_{22}
	\end{pmatrix}\right\|, $$
	where $||\cdot ||$ is the matrix $L_2$ norm. 
\end{lemma}
\begin{proof}[Proof]
	Note
	\bea
	\left\|  \begin{pmatrix}
		A_{11} & A_{12} \\
		A_{12}^T & A_{22}
	\end{pmatrix}\right\| 
	&=& \sup_{\|x\|_2=1} \left\| \begin{pmatrix}
		A_{11} & A_{12} \\ A_{12}^T & A_{22}
	\end{pmatrix}x  \right\|_2 \\ 
	&=& \sup_{\|x\|_2=1} \left\| \binom{A_{11}x_1+A_{12}x_2}{A_{22} x_2+A_{12}^T x_1} \right\|_2 \\
	&\ge& \sup_{\|x_2\|_2=1} \left\| \binom{A_{12}x_2}{A_{22} x_2} \right\|_2 \,\,\ge\,\,  \sup_{\|x_2\|_2=1} \| A_{12} x_2\|_2 = \| A_{12} \|
	\eea
	where $x=(x_1^T, x_2^T)^T$ and $x_1 \in \bbR^{p_1}, x_2\in \bbR^{p_2}$. This completes the proof. \quad $\blacksquare$
\end{proof}

\begin{lemma}\label{diffAbound}
	If we assume that $\Omega_{0,n} \in \cU(\epsilon_0,\gamma)$, then
	\bea
	\| A_{0,nk} - A_{0,n} \|_\infty &\le& C \sqrt{k}\gamma(k),\\
	\| A_{0,nk} - A_{0,n} \|_1 &\le& C k\gamma(k)
	\eea
	for some $C>0$.
\end{lemma}
\begin{proof}[Proof of Lemma \ref{diffAbound}]
	We only consider $k<j-1$ case because $A_{0,nk} = A_{0,n}$ trivially holds when $k\ge j-1$.
	Note first that 
	\bea
	\| A_{0,nk} - A_{0,n} \|_\infty &\le& \| A_{0,nk} - B_k(A_{0,n}) \|_\infty + \| B_k(A_{0,n}) - A_{0,n} \|_\infty.
	\eea
	The second term is bounded above by $\gamma(k)$ by the definition of $\cU(\epsilon_0,\gamma)$. Denote
	\bea
	\V^{-1}(Z_j) &=& \begin{pmatrix}
		\Omega_{11,j} & \Omega_{12,j}\\
		\Omega_{21,j} & \Omega_{22,j}
	\end{pmatrix}, \\
	\C( Z_j , X_j) &=& \binom{\sg_{1j}}{\sg_{2j}},
	\eea
	where $\Omega_{11,j}$ is a $(j-k-1)\times (j-k-1)$ matrix, $\Omega_{22,j}$ is a $k\times k$ matrix and $\sg_{2j} =\C(Z_j^{(k)}, X_j)$ is a $k$ dimensional vector. 
	Since $\max_j\|a_{0,j} - B_{k-1,j}(a_{0,j})\|_1 \le \gamma(k)$ by assumption, it directly implies 
	\bea
	\max_j \| \Omega_{11,j}\sg_{1j} + \Omega_{12,j}\sg_{2j}\|_1 &\le& \gamma(k).
	\eea
	With this fact, we have the following upper bound for $\|A_{0,nk} - B_k(A_{0,n}) \|_{\infty}$,
	\bea
	\|A_{0,nk} - B_k(A_{0,n}) \|_{\infty} &=& 
	\max_j\|a_{0,j}^{(k)} - B_{k-1,j}(a_{0,j})\|_{1} \\
	&=& \max_j  \|\Omega_{21,j}\Omega_{11,j}^{-1} ( \Omega_{11,j}\sg_{1j}+\Omega_{12,j}\sg_{2j})\|_{1} \\
	&\le& \max_j \|\Omega_{21,j}\Omega_{11,j}^{-1}\|_{1} \|\Omega_{11,j}\sg_{1j}+\Omega_{12,j}\sg_{2j}\|_1 \\
	&\le& \max_j \sqrt{k}\|\Omega_{21,j}\| \|\Omega_{11,j}^{-1}\| \gamma(k) \\
	&\le& C \sqrt{k}\gamma(k)
	\eea
	for some $C>0$. 
	The last inequality holds by Lemma \ref{spec_partitioned} because $\|\V(Z_j)^{-1}\| \le C$ holds for some $C>0$. It proves the first part of Lemma \ref{diffAbound}. 	
	
	To show the second argument of Lemma \ref{diffAbound}, note that
	\bea
	\| A_{0,nk} - A_{0,n} \|_1 &\le& \| A_{0,nk} - B_k(A_{0,n}) \|_1 + \| B_k(A_{0,n}) - A_{0,n} \|_1.
	\eea 
	The first term is bounded above by
	\bea
	\| A_{0,nk} - B_k(A_{0,n}) \|_1 &\le& k \max_j \|a_{0,j}^{(k)} - B_{k-1,j}(a_{0,j})\|_{\max} \\
	&\le& k \max_j \|\Omega_{21,j}\Omega_{11,j}^{-1}\| \|\Omega_{11,j}\sg_{1j}+\Omega_{12,j}\sg_{2j}\|_2 \\
	&\le& C k \gamma(k)
	\eea
	for some $C>0$. Also note that 
	\bea
	\|B_k(A_{0,n}) - A_{0,n} \|_1 &=& \sum_{i=j+k}^p |a_{ij}| \\
	&\le& \sum_{i=j+k}^p \gamma(i-j) \\
	&\le& \sum_{m=k}^\infty \gamma(m).
	\eea
	If we assume the polynomially decreasing $\gamma(k)= Ck^{-\alpha}$, we have $\sum_{m=k}^\infty \gamma(m) \le C' k \gamma(k)$ for some constant $C'>0$.
	If we assume the exact band or exponentially decreasing $\gamma(k) = Ce^{-\beta k}$, it is easy to show that $\sum_{m=k}^\infty \gamma(m) \le C' \gamma(k)$ for some constant $C'>0$. Thus, $\|A_{0,nk} - A_{0,n}\|_1$ is bounded above by $C''k\gamma(k)$ for some constant $C''>0$. \quad $\blacksquare$
\end{proof}

\bibliographystyle{plain}
\bibliography{cholcov}

\end{document}